\documentclass[12pt]{article}
\usepackage{amsmath}
\usepackage{latexsym}
\usepackage{amssymb}
\newtheorem{thm}{Theorem}[section]
\newtheorem{la}[thm]{Lemma}
\newtheorem{Defn}[thm]{Definition}
\newtheorem{Remark}[thm]{Remark}
\newtheorem{Note}[thm]{Note}
\newtheorem{prop}[thm]{Proposition}

\newtheorem{cor}[thm]{Corollary}
\newtheorem{Example}[thm]{Example}
\newtheorem{Examples}[thm]{Examples}
\newtheorem{Problems}[thm]{Problems}

\newtheorem{Problem}[thm]{Problem}
\newtheorem{Convention}[thm]{Convention}
\newtheorem{Number}[thm]{\!\!}

\newenvironment{rem}{\begin{Remark}\rm}{\end{Remark}}
\newenvironment{numba}{\begin{Number}\rm}{\end{Number}}
\newenvironment{proof}{{\noindent\bf Proof.}}%
                  {\nopagebreak\hspace*{\fill}$\Box$\medskip\medskip\par}   
\newcommand{\Punkt}{\nopagebreak\hspace*{\fill}$\Box$}
\newcommand{\wb}{\overline}
\newcommand{\ve}{\varepsilon}
\newcommand{\at}{\symbol{'100}}

\newcommand{\wt}{\widetilde}

\newcommand{\mto}{\mapsto}

\newcommand{\isom}{\cong}
\DeclareMathOperator{\Ad}{Ad}
\newcommand{\N}{{\mathbb N}}
\newcommand{\R}{{\mathbb R}}

\newcommand{\K}{{\mathbb K}}

\newcommand{\C}{{\mathbb C}}

\DeclareMathOperator{\ad}{ad}

\newcommand{\cO}{{\cal O}}

\newcommand{\cg}{{\mathfrak g}}

\newcommand{\dl}{{\displaystyle \lim_{\longrightarrow}}}

\newcommand{\wh}{\widehat}

\newcommand{\sub}{\subseteq}

\DeclareMathOperator{\im}{im}

\DeclareMathOperator{\id}{id}

\newcommand{\cA}{{\cal A}}

\newcommand{\cR}{{\cal R}}

\newcommand{\cT}{{\cal T}}

\DeclareMathOperator{\Repart}{Re}

\DeclareMathOperator{\Supp}{supp}

\begin{document}
$\;$\\[-27mm]
\begin{center}
{\Large\bf Continuity of LF-algebra representations\\[2mm]
associated to representations of Lie groups}\\[3mm]
{\bf Helge Gl\"{o}ckner}\footnote{Supported by DFG,
project GL 357/5-2.}
\end{center}
\begin{abstract}\vspace{.3mm}\noindent
Let $G$ be a finite-dimensional Lie group
and $E$ be a locally convex
topological $G$-module. If $E$ is sequentially complete,
then $E$ and the space $E^\infty$ of smooth
vectors are
$C^\infty_c(G)$-modules, but the module\linebreak
multiplication
need not be continuous.
The pathology can be ruled out if~$E$ is (or embeds into) a projective limit of
Banach $G$-modules.
Moreover, in this case
$E^\omega$ (the space of analytic vectors)
is a module for the algebra $\cA(G)$ of
superdecaying analytic functions
introduced\linebreak
by Gimperlein, Kr\"{o}tz and Schlichtkrull.
We prove that $E^\omega$ is a\linebreak
\emph{topological} $\cA(G)$-module
if $E$ is a Banach space or, more generally,
if every countable set
of continuous seminorms on~$E$ has an upper bound.
The same conclusion is obtained if $G$ has a compact Lie algebra.
The question of whether
$C^\infty_c(G)$ and $\cA(G)$
are topological algebras is also addressed.\vspace{.5mm}
\end{abstract}
{\footnotesize {\em Classification}:
22D15, 46F05 (Primary); 22E30, 42A85, 46A13, 46E25.\\[1mm]
%
%
%22D15,
% group algs of lcp gps
%
%46F05 (Primary);
% test functions
%
%22E30,
% ana on real and complex Lie groups
%
%42A85,
% convolution classical
%
%46A13,
% proj and ind limits of tvs
%
%46E25.\\[1mm]
% rings and algebras of functions
%
%
}\\[6mm]
\noindent
{\bf\large Introduction and statement of results}\\[3mm]
We study continuity properties of
algebra actions associated with representations of a (finite-dimensional, real) Lie group~$G$.
Throughout this note, $E$ denotes a topological $G$-module,
i.e., a complex locally convex space endowed
with a continuous left $G$-action $\pi\colon G\times E\to E$
by linear maps $\pi(g,.)$.\\[3mm]
{\bf Results concerning {\boldmath$C^\infty_c(G)$} and the space of smooth vectors}\\[3mm]
Our first results concern the convolution algebra $C^\infty_c(G)$ of complex-valued
test functions on a Lie group~$G$.
As usual, $v\in E$ is called a \emph{smooth vector}
if the orbit map
$\pi_v\colon G\to E$, $\pi_v(g):=\pi(g,v)$ is smooth.
The space $E^\infty$ is\linebreak\vspace{-2mm}\pagebreak

\noindent
endowed with the initial
topology $\cO_{E^\infty}$ with respect to the map
\begin{equation}\label{defPhi}
\Phi\colon E^\infty\to C^\infty(G,E), \quad \Phi(v)=\pi_v\,.
\end{equation}
Let $\lambda_G$ be a left Haar measure on~$G$.
If $E$ is sequentially complete or has the metric convex compactness
property (see \cite{Voi} for information on this concept),\footnote{That is, each metrizable compact subset $K\sub E$ has a relatively compact
convex hull.} then the weak integral
\begin{equation}\label{theint}
\Pi(\gamma,v):=\int_G\gamma(x)\pi(x,v)\,d\lambda_G(x)
\end{equation}
exists in~$E$ for all $v\in E$ and $\gamma\in C^\infty_c(G)$
(see \cite[1.2.3]{Her} and \cite[3.27]{Rud}).
In this way, $E$ becomes a $C^\infty_c(G)$-module.
Moreover, $\Pi(\gamma,v)\in E^\infty$ for all
$\gamma\in C^\infty_c(G)$ and $v\in E$, whence
$E^\infty$ is a $C^\infty_c(G)$-submodule (as we recall in Lemma~\ref{csuppE}).
It is natural to ask whether the module multiplication
\begin{eqnarray}
C^\infty_c(G)\times E \;\;\, & \to &  \,E,\quad \;\;(\gamma,v)\mto\Pi(\gamma,v)\;\;\;\;\mbox{resp.}\label{disc1}\\
C^\infty_c(G)\times E^\infty &\to & E^\infty,\quad (\gamma,v)\mto \Pi(\gamma,v)\label{disc2}
\end{eqnarray}
is continuous, i.e., if $E$ and $(E^\infty,\cO_{E^\infty})$
are topological $C^\infty_c(G)$-modules.
Contrary to a recent assertion \cite[pp.\,667--668]{Kro},
this can fail even if $E$ is Fr\'{e}chet:\\[4.5mm]
{\bf Proposition~A.}
\emph{If $G$ is a non-compact Lie group
and $E:=C^\infty(G)$ with $\pi\colon G\times C^\infty(G)\to C^\infty(G)$,
$\pi(g,\gamma)(x):=\gamma(g^{-1}x)$, then neither $E$ nor $E^\infty$
are topological $C^\infty_c(G)$-modules,
i.e., the maps} (\ref{disc1}) and (\ref{disc2})
\emph{are discontinuous.}\\[4.5mm]
A continuous seminorm $p$ on~$E$
is called \emph{$G$-continuous} if 
$\pi\colon G\times (E,p)\to (E,p)$ is continuous~\cite[p.\,7]{BaK}.
Varying terminology from~\cite{HaM},
we call a topological $G$-module~$E$ \emph{proto-Banach}
if the topology of~$E$ is defined by a set of $G$-continuous seminorms.\footnote{Viz.\
$E$ embeds into a projective limit of Banach $G$-modules, cf.\ \cite[Remark~2.5]{BaK}.}
If $E$ is a Fr\'{e}chet space, then $E$
is proto-Banach if and only if there is
a sequence $(p_n)_{n\in\N}$
of $G$-continuous seminorms defining the topology,
i.e., if and only if  $\Pi$ is an F-representation
as in~\cite{BaK}, \cite{Kro},~\cite{GKS}.\\[4mm]
{\bf Proposition~B.}
\emph{Let $G$ be a Lie group and $E$ a proto-Banach $G$-module
that is sequentially complete or has the metric convex compactness
property. Then
the map $\Pi\colon C^\infty_c(G)\times E\to E^\infty$ from}~(\ref{theint}) \emph{is continuous.
In particular, $E$~and~$E^\infty$ are topological $C^\infty_c(G)$-modules.}\\[4mm]
We mention that $C^\infty_c(G)$ is a topological
algebra if and only if $G$ is $\sigma$-compact~\cite[p.\,3]{PJM}
(cf.\ \cite[Proposition~2.3]{Hir} for the special case $G=\R^n$).\\[4mm]
{\bf Results concerning {\boldmath$\cA(G)$} and the space of analytic vectors}\\[2.7mm]
Let $G$ be a connected Lie group now.
If $E$ is a topological $G$-module,
say that $v\in E$ is an \emph{analytic vector}
if the orbit map $\pi_v\colon G\to E$ is real analytic
(in the sense recalled in Section~\ref{preC}).
Write $E^\omega\sub E$ for the space of all analytic vectors.
If $G\sub G_\C$ (which we assume henceforth for simplicity of the presentation),
let $(V_n)_{n\in \N}$ be a basis
of relatively compact, symmetric, connected
identity neighbourhoods
in~$G_\C$, such that $V_n\supseteq \wb{V_{n+1}}$
(e.g., we can choose $V_n$
as in \cite{GKS}).
Then $v\in E$ is an analytic vector
if and only if $\pi_v$ admits a complex analytic
extension $\wt{\pi}_v\colon GV_n\to E$ for some $n\in \N$
(see Lemma~\ref{extGC}).
We write $E_n\sub E^\omega$
for the space of all $v\in E^\omega$
such that $\pi_v$ admits a $\C$-analytic extension to $GV_n$,
and give $E_n$ the topology making
\[
\Psi_n\colon E_n\to \cO(GV_n,E),\quad v\mto \wt{\pi}_v
\]
a topological embedding,
using the compact-open topology on
the space $\cO(GV_n,E)$ of all $E$-valued $\C$-analytic maps on~$GV_n$.
Like~\cite{GKS},
we give $E^\omega$ the topology making it
the direct limit $E^\omega=\dl\, E_n$\vspace{-.3mm} as a locally convex space.\footnote{If $G$ is an arbitrary connected Lie group,
let $q\colon \wt{G}\to G$ be the universal covering group
and $V_n\sub (\wt{G})_\C$ as above. Then $\wt{G}\sub (\wt{G})_\C$.
Define $E_n$ now as the space of all
$v\in E^\omega$ such that $\pi_v\circ q$
has a complex analytic extension to $\wt{G}V_n\sub (\wt{G})_\C$,
and topologize $E^\omega$ as before. In this way, we could easily drop the condition that $G\sub G_\C$.}\\[2.3mm]
We fix a left invariant
Riemannian metric ${\bf g}$ on~$G$, let ${\bf d} \colon G\times G\to [0,\infty[$
be the associated left invariant distance function,
and set
\begin{equation}\label{defnd}
d(g):={\bf d}(g,1)\quad\mbox{for $\,g\in G$.}
\end{equation}
Following \cite{GKS} and~\cite{Kro},
we let $\cR(G)$ be the Fr\'{e}chet space
of continuous functions $\gamma\colon G\to\C$
which are \emph{superdecreasing}
in the sense that
\begin{equation}\label{reusno}
\|\gamma\|_N:=\sup\{|\gamma(x)|e^{Nd(x)}\colon x\in G\}<\infty\quad\mbox{for all $\,N\in \N_0$.}
\end{equation}
Then $\cR(G)$ is a topological algebra under convolution
\cite[Proposition~4.1\,(ii)]{GKS}.
If $E$ is a sequentially complete proto-Banach $G$-module,
then
\[
\Pi(\gamma,v):=\int_G\gamma(x)\pi(x,v)\,d\lambda_G(x)\quad\mbox{for $\gamma\in \cR(G)$, $v\in E$}
\]
exists in~$E$ as an absolutely convergent integral,
and $\Pi$ makes $E$ a topological $\cR(G)$-module
(like for F-representations, \cite[Proposition~4.1\,(iii)]{GKS}).\\[3mm]
As $G\times \cR(G)\to \cR(G)$, $\pi(g,\gamma)(x):=\gamma(g^{-1}x)$
is an F-representation \cite[Prop.\,4.1\,(i)]{GKS}, $\cA(G):=\cR(G)^\omega$
is the locally convex direct limit of the steps $\cA_n(G):=\cR(G)_n$.
Since $\C$-analytic extensions of orbit maps
can be multiplied pointwise in $(\cR(G),*)$, both
$\cA_n(G)$ and $\cA(G)$ are subalgebras of~$\cR(G)$.\\[4mm]
If $E$ is a sequentially complete proto-Banach $G$-module,
then
\begin{equation}
\Pi(\gamma,v)\in E^\omega\quad\mbox{for all $\gamma\in \cA(G)$, $v\in E$; moreover,}
\end{equation}
\begin{equation}\label{mpxx}
\cA_n(G)\times E\to E_n, \;\;
(\gamma,v)\mto \Pi(\gamma,v)\quad
\mbox{is continuous for each $\,n\in \N$.}
\end{equation}
This can be shown
as in the case of F-representations in \cite[Proposition~4.6]{GKS}.\\[3mm]
{\bf Problem.}
The following assertions concerning F-representations
and the algebras $\cA(G)$ (stated in \cite[Propositions~4.2\,(ii) and 4.6]{GKS})
seem to be open in general (in view of difficulties
explained presently, in Remark~2):
\begin{itemize}
\item[(a)]
Is $\Pi\colon \cA(G)\times E\to E^\omega$
continuous for each F-representation $(E,\pi)$
(or even for each sequentially complete proto-Banach $G$-module)?\footnote{By Lemma~\ref{reach},
$\Pi\colon \cA(G)\times E\to E^\omega$ is always separately continuous, hypocontinuous in its second argument,
and sequentially continuous (hence also the maps in (b) and (c)).}
\item[(b)]
Is $\Pi\colon \cA(G)\times E^\omega\to E^\omega$
continuous in the situation of (a)?\footnote{(b) follows from (a)
as the inclusion $E^\omega\to E$ is continuous linear.}
\item[(c)] Is convolution $\cA(G)\times \cA(G)\to\cA(G)$ continuous?
\end{itemize}
To formulate a solution to these problems in special cases,
recall that a pre-order
on the set $P(E)$ of all continuous seminorms $p$
on a locally convex space~$E$ is obtained by declaring
$p\preceq q$ if $p\leq Cq$ pointwise for some $C>0$.
The space $E$ is said to have the \emph{countable neighbourhood property}
if every countable set $M\sub P(E)$ has an upper bound in
$(P(E),\preceq)$ (see \cite{Bon}, \cite{Flo}
and the references therein).\\[4mm]
{\bf Proposition~C.}
\emph{Let $G$ be a connected Lie group with $G\sub G_\C$
and $E$ be a sequentially complete, proto-Banach $G$-module.
If $E$ is normable or $E$ has the countable neighbourhood property,
then $\Pi\colon \cA(G)\times E\to E^\omega$
is continuous.
In particular, $E^\omega$ is a topological $\cA(G)$-module.}\\[3mm]
{\bf Remark 1.}
Recall that a metrizable locally convex space has the countable neighbourhood property (c.n.p.)
if and only if it is normable.
Because the c.n.p.\ is inherited by
countable locally convex direct limits~\cite{Flo},
it follows that every LB-space
(i.e., every countable locally convex direct limit of Banach spaces)
has the c.n.p.
Also locally convex spaces~$E$
which are $k_\omega$-spaces have the c.n.p.\
(see \cite[Corollary~8.1]{UBC}
and \cite[Example~9.4]{JFA}; cf.\ \cite{Bon}).\footnote{A topological space $X$ is $k_\omega$
if $X\!=\!\dl\,K_n$\vspace{-.3mm} with compact spaces $K_1\sub K_2\sub\cdots$\,\cite{FSM},\cite{GHK}.}
For example, the dual~$E'$ of any metrizable topological vector
space~$E$ is a $k_\omega$-space, when equipped with the compact-open topology
(cf.\ \cite[Corollary~4.7]{AUS}).\\[4mm]
For $G$ a compact, connected Lie group,
the convolution
$\cA(G)\times \cA(G)\to \cA(G)$
is continuous and thus $\cA(G)$
is a topological algebra.
In fact, $\cR(G)=C(G)$ is normable in this case (as each $\|.\|_N$
is equivalent to $\|.\|_\infty$ then). Since $\cA(G)=\cR(G)^\omega$,
Proposition~C applies.\\[4mm]
The same conclusion can be obtained
by an alternative argument,
which shows also that $(\cA(G),*)$
is a topological algebra for each abelian connected
Lie group~$G$.
In contrast to the setting of Proposition~C,
quite general spaces~$E$ are allowed now,
but conditions are imposed on~$G$.
Recall that a real Lie algebra $\cg$ is said to be \emph{compact}
if it admits an inner product making $e^{\ad(x)}$
an isometry for each $x\in \cg$
(where $\ad(x):=[x,.]$ as usual).
If $G$ is compact or abelian, then its Lie algebra $L(G)$
is compact.\\[4mm]
{\bf Proposition~D.}
\emph{Let $G$ be a connected Lie group with $G\sub G_\C$,
whose Lie algebra $L(G)$ is compact.
Then $E^\omega$ is a topological $\cA(G)$-module,
for each sequentially complete, proto-Banach $G$-module~$E$.
In particular, convolution is jointly continuous and
thus $(\cA(G),*)$ is a topological algebra.}\\[3mm]
{\bf Remark 2.}
Note that, due to the continuity of the maps (\ref{mpxx}), the map
\[
\Pi\colon \cA(G)\times E\to E^\omega
\]
is continuous with respect to the topology $\cO_{DL}$
on $\cA(G)\times E$ which makes it the direct limit $\dl\,(\cA_n(G)\times E)$\vspace{-.3mm}
as a topological space.
On the other hand, there is the topology $\cO_{lcx}$
making $\cA(G)\times E$ the direct
limit $\dl\,(\cA_n(G)\times E)$\vspace{-.3mm}
as a locally convex space.
Since locally convex direct limits and two-fold products commute \cite[Theorem~3.4]{Hir},
$\cO_{lcx}$ coincides with the product topology on $(\dl\,\cA_n(G))\times E=\cA(G)\times E$
and hence is the topology we are interested in.
Unfortunately, as $\Pi$ is not linear,
it is \emph{not} possible to deduce continuity of $\Pi$ on $(\cA(G)\times E,\cO_{lcx})$
from the continuity of the maps~(\ref{mpxx}).\footnote{This problem
was overlooked in \cite[proof of Proposition~4.6]{GKS}.}$'$\footnote{Note that, in Proposition~A,
the convolution $C^\infty_c(\R)\times C^\infty(\R)\to C^\infty(\R)$ is discontinuous, although its restriction to
$C^\infty_{[-n,n]}(\R)\times C^\infty(\R)\to C^\infty(\R)$ is continuous for all~$n\in \N$.}\\[2.4mm]
Of course, whenever $\cO_{DL}=\cO_{lcx}$,
we obtain continuity of $\Pi\colon \cA(G)\times E\to E^\omega$
with respect to~$\cO_{lcx}$.
Now $\cO_{lcx}\sub \cO_{DL}$ always,
but equality $\cO_{lcx}=\cO_{DL}$ only occurs in exceptional
situations.
% e.g.\ if $G$ is compact and $E$ a $k_\omega$-space. CHECK
In the prime case of an F-representation $(E,\pi)$ of~$G$,
we have $\cO_{DL}\not={\cO}_{lcx}$ in all cases of interest, as we shall presently see.
Thus, Problems\,(a)--(c)
remain open in general (apart from the special cases settled in Propositions~C and~D).\\[4mm]
The following observation pinpoints the source of these difficulties.\\[3mm]
{\bf Proposition E.}
\emph{If
$E$ is an infinite-dimensional
Fr\'{e}chet space
and $G$ a connected Lie group with $G\sub G_\C$
and $G\not=\{1\}$,
then $\cO_{DL}\not =\cO_{lcx}$
on $\cA(G)\times E$.}\\[3mm]
Proposition~E will be deduced from a new result on direct limits of topological\linebreak
groups (Proposition~\ref{yamprop}), which is a variant of Yamasaki's Theorem\linebreak
\cite[Theorem~4]{Yam}
for direct sequences which need not be strict, but
are\linebreak
sequentially compact regular.\\[2.4mm]
We mention that $\cA(G)$ also
is an algebra under \emph{pointwise} multiplication (instead of convolution),
and in fact a topological algebra (see Section~\ref{pointmult}).
%In Section~\ref{secconv},
%we explain differences which prevent an analogous treatment
%of~$(\cA(G),*)$.
Sections~\ref{secprel}--\ref{secpropB}
are devoted to Propositions~A and~B;
Sections~\ref{preC}--\ref{secpropE}
are devoted to the proofs of Propositions~C, D and~E
(and the respective preliminaries). The proofs of some auxiliary
results have been relegated to the appendix.
See also \cite{Ne2}, \cite{Ne3}
for recent studies of smooth and analytic vectors
(with a view towards infinite-dimensional groups).\\[4mm]
{\bf Basic notations.}
We write $\N_0:=\{0,1,2,\ldots\}$.
If $E$ is a vector space and $q$ a seminorm on~$E$,
we set
$B^q_\ve(x):=\{y\in E\colon q(y-x)<\ve\}$
for $x\in E$, $\ve>0$.\linebreak
$L(G)$ is the Lie algebra of a Lie group $G$
and $\im(f)$ the image of a map $f$.
If $X$ is a set and $f\colon X\to\C$ a map,
as usual $\|f\|_\infty:=\sup\{|f(x)|\colon x\in X\}$.
If $q$ is a seminorm on a vector space $E$ and
$f\colon X\to E$, we write $\|f\|_{q,\infty}:=\sup\{q(f(x))\colon x\in X\}$.
\section{Preliminaries for Propositions~A and~B}\label{secprel}
We shall use concepts and basic tools from
calculus in locally convex spaces.
\begin{numba}
Let $E$, $F$ be real locally convex spaces,
$U\sub E$ open, $r\in \N_0\cup\{\infty\}$.
Call $f\colon U\to F$ a $C^r$-map if~$f$ is continuous,
the iterated directional derivatives
\[
d^{(k)}f(x,y_1,\ldots, y_k):=(D_{y_k}\cdots D_{y_1}f)(x)
\]
exist in $E$ for all $k\in \N_0$ such that $k\leq r$,
$x\in U$ and $y_1,\ldots, y_k\in E$, and each
$d^{(k)}f\colon U\times E^k \to F$ is continuous.
The $C^\infty$-maps are also called \emph{smooth}.
\end{numba}
See~\cite{RES},
\cite{GaN}, \cite{Ham}, \cite{Mic}, \cite{Mil}. Since
compositions of $C^r$-maps are $C^r$, one can
define $C^r$-manifolds modelled on locally convex spaces
as expected.
\begin{numba}
Given a Hausdorff space~$M$ and locally convex space~$E$,
we endow the space $C^0(M,E)$ of continuous $E$-valued functions on~$M$
with the compact-open topology.
If $M$ is a $C^r$-manifold modelled on a locally convex space~$X$,
we give $C^r(M,E)$ the compact-open $C^r$-topology, i.e.,
the initial topology with respect to the maps $C^r(M,E)\to C^0(V\times X^k,E)$, $\gamma\mto d^{(j)}(\gamma\circ \phi^{-1})$
for all charts $\phi\colon U\to V$ of~$M$.
If $M$ is finite-dimensional and $K\sub M$ compact,
as usual we endow $C^r_K(M,E):=\{\gamma\in C^r(M,E)\colon \gamma|_{M\setminus K}=0\}$
with the topology induced by $C^r(M,E)$,
and give $C^r_c(M,E)=\bigcup_K\, C^r_K(M,E)$\vspace{-.3mm} the locally convex direct
limit topology. We abbreviate $C^r(M):=C^r(M,\C)$, etc.
\end{numba}
The following variant is essential for our purposes.
\begin{numba}\label{RS}
Let $E_1$, $E_2$ and $F$ be real locally convex spaces,
$U\sub E_1\times E_2$ be open and $r,s\in \N_0\cup\{\infty\}$.
A map $f\colon U\to F$ is called a \emph{$C^{r,s}$-map}
if the derivatives
\[
d^{(k,\ell)}f(x,y,a_1,\ldots, a_k,b_1,\ldots,b_\ell):=
(D_{(a_k,0)}\cdots D_{(a_1,0)}D_{(0,b_\ell)}\cdots D_{(0,b_1)}f)(x,y)
\]
exist for all $k,\ell\in \N_0$ such that $k\leq r$ and $\ell\leq s$,
$(x,y)\in U$ and $a_1,\ldots, a_k\in E_1$, $b_1,\ldots, b_\ell\in E_2$,
and $d^{(k,\ell)}f \colon U\times E_1^k\times E_2^\ell\to F$ is continuous.
\end{numba}
See \cite{Alz} and \cite{AaS} for a detailed
development of the theory of $C^{r,s}$-maps.
Notably, $f$ as in~(\ref{RS}) is $C^{\infty,\infty}$
if and only if it is smooth.
If $h\circ f\circ (g_1\times g_2)$ is defined,
where $h$ is $C^{r+s}$, $f$ is $C^{r,s}$,
$g_1$ is $C^r$ and $g_2$ is $C^s$, then the map
$h\circ f\circ (g_1\times g_2)$ is $C^{r,s}$.
As a consequence, we can speak of
$C^{r,s}$-maps
$f\colon M_1\times M_2\to M$
if $M,M_1,M_2$ are smooth manifolds (likewise for
$f\colon U\to M$ on an open set $U\sub M_1\times M_2$).
See op.\,cit.\ for these basic facts,
as well as the following aspect of the exponential law for $C^{r,s}$-maps,
which is essential for us.\footnote{Exponential laws for smooth
functions are basic tools of infinite-dimensional analysis;
see, e.g., \cite{ZOO} (compare \cite{KaM} for related bornological
results).}
\begin{la}\label{explaw}
Let $r,s\in \N_0\cup\{\infty\}$, $E$ be a locally convex space,
$M$ be a $C^r$-manifold
and $N$ be a $C^s$-manifold $($both
modelled on some locally convex space$)$.
If $f\colon M\times N\to E$ is a $C^{r,s}$-map, then
\[
f^\vee(x)\colon M\to C^s(N,E),\quad
f^\vee(x)(y):=f(x,y)
\]
is a $C^r$-map.
Hence, if $g\colon M\to C^s(N,E)$ is a map such that
$\wh{g}\colon M\times N\to E$, $\wh{g}(x,y):=g(x)(y)$
is $C^{r,s}$, then~$g$ is~$C^r$.\Punkt
\end{la}
In particular, we shall encounter $C^{\infty,0}$-maps of the following form.
\begin{la}\label{specsit}
Let $E_1$, $E_2$, $E_3$ and $F$ be locally convex spaces, $U_1\sub E_1$ and $U_2\sub E_2$
be open,
$g\colon U_1\times U_2\to\C$ be a smooth map,
$h\colon U_1\to E_3$ be a smooth map
and $\pi\colon U_2\times E_3\to F$
be a continuous map such that $\pi(y,.)\colon E_3\to F$ is linear
for each $y\in U_2$.
Then the following map is $C^{\infty,0}$:
\[
f\colon U_1\times U_2\to F,\quad
f(x,y):=g(x,y)\pi(y,h(x)).
\]
\end{la}
For the proof of Lemma~\ref{specsit}
(and those of the next four lemmas), the reader is referred to
Appendix~A.
\begin{la}\label{lemB1}
For each Lie group $G$, the left translation action
\[
\pi\colon
G\times C^\infty_c(G)\to C^\infty_c(G),\quad \pi(g,\gamma)(x):=\gamma(g^{-1}x)
\]
is a smooth map.
\end{la}
We mention that $G$ is not assumed $\sigma$-compact in Lemma~\ref{lemB1}
(of course, $\sigma$-compact groups are the case of primary interest).
\begin{la}\label{lemB2}
Let $X$ be a locally compact space, $E$ be a locally convex space
and $f\in C^0(X,E)$.
Then the multiplication operator
$m_f\colon C^0_c(X)\to C^0_c(X,E)$, $m_f(\gamma)(x):=\gamma(x)f(x)$
is continuous linear.
\end{la}
We also need a lemma on the parameter dependence of weak integrals.
Note that the definition of $C^{r,0}$-maps does not use the vector space structure
on~$E_2$, and makes perfect sense if~$E_2$
is merely a topological space.
\begin{la}\label{pardep}
Let $X$, $E$ be locally convex spaces, $P\sub X$ open, $r\in \N_0\cup\{\infty\}$,
$K$ be a compact topological space, $\mu$ a finite measure
on the $\sigma$-algebra of Borel sets of~$K$
and $f\colon P\times K\to E$ be a $C^{r,0}$-map.
Assume that the weak integral
$g(p):=\int_Kf(p,x)\,d\mu(x)$
exists in~$E$, as well as the weak integrals
\begin{equation}\label{candder}
\int_Kd^{(k,0)}f(p,x,q_1,\ldots,q_k)\,d\mu(x)\,,
\end{equation}
for all $k\in \N$ such that $k\leq r$, $p\in P$ and $q_1,\ldots, q_k\in X$.
Then $g\colon P\to E$ is a $C^r$-map, with $d^{(k)}g(p,q_1,\ldots, q_k)$
given by~{\rm(\ref{candder})}.
\end{la}
\begin{la}\label{csuppE}
Let $G$ be a Lie group and $\pi\colon G\times E\to E$
be a topological $G$-module which is sequentially complete
or has the metric convex compactness property.
Then $w:=\Pi(\gamma,v)\in E^\infty$
for all $\gamma\in C^\infty_c(G)$ and $v\in E$.
In particular, $E$ and $E^\infty$ are $C^\infty_c(G)$-modules.
\end{la}
\section{Proof of Proposition~A}\label{secpropA}
The evaluation map $\ve\colon C^\infty(G)\times G\to\C$, $(\gamma,x)\mto\gamma(x)$
is smooth (see, e.g., \cite{GaN} or \cite[Proposition~11.1]{ZOO}).
In view of Lemma~\ref{explaw}, the mapping\linebreak
$\pi\colon G\times C^\infty(G)\to C^\infty(G)$, $\pi(g,\gamma)(x)=\gamma(g^{-1}x)$
is smooth, because
\[
\wh{\pi}\colon G\times C^\infty(G)\times G\to \C,\quad
(g,\gamma,x)=\gamma(g^{-1}x)=\ve(\gamma,g^{-1}x)
\]
is smooth.
Hence each $\gamma\in C^\infty(G)$ is a smooth vector.
Using Lemma~\ref{explaw} again,
we see that the linear map
\[
\Phi\colon C^\infty(G)\to C^\infty(G,C^\infty(G)),\quad \Phi(\gamma)=\pi_\gamma
\]
is smooth (and hence continuous) because
$\wh{\Phi}\colon C^\infty(G)\times G\to C^\infty(G)$,
$\wh{\Phi}(\gamma,g)=\pi_\gamma(g)=\pi(g,\gamma)$
is smooth. As a consequence, $C^\infty(G)$ and the space $C^\infty(G)^\infty$
of smooth vectors coincide as locally convex spaces.\\[2.7mm]
Now $\Pi(\gamma,\eta)=\gamma*\eta$
for $\gamma\in C^\infty_c(G)$ and $\eta\in C^\infty(G)$,
In fact, for each $x\in G$, the point evaluation $\ve_x\colon C^\infty(G)\to\C$, $\theta\mto\theta(x)$
is continuous linear. Hence
\[
\Pi(\gamma,\eta)(x)\!=\!\Big(\!\int_G\!\gamma(y)\eta(y^{-1}\cdot)\,d\lambda_G(y)\Big)(x)\!=\!\!
\int_G\!\gamma(y)\eta(y^{-1}x)\,d\lambda_G(y)\!=\!(\gamma*\eta)(x).
\]
Thus $\Pi$ is the map
$C^\infty_c(G)\times C^\infty(G)\to C^\infty(G)$, $(\gamma,\eta)\mto\gamma*\eta$,
which is discontinuous by~\cite[Proposition~7.1]{PJM}
(or the independent work \cite{Lar},
in the special case $G=\R^n$).
\section{Proof of Proposition~B}\label{secpropB}
\begin{la}\label{hypoetc}
In the situation of Lemma~{\rm\ref{csuppE}},
the bilinear mapping\linebreak
$\Pi\colon C^\infty_c(G)\times E\to E^\infty$
is separately continuous, hypocontinuous in its\linebreak
second argument
and sequentially continuous.
If $E$ is barrelled $($e.g., if $E$ is a Fr\'{e}chet space$)$,
then $\Pi$ is hypocontinuous in both arguments.
\end{la}
\begin{proof}
We need only show that $\Pi$ is separately continuous.
In fact, $C^\infty_c(G)$ is barrelled,
being a locally convex direct limit of Fr\'{e}chet spaces \cite[II.7.1 and II.7.2]{SaW}.
Hence, if $\Pi$ is separately continuous, it
automatically is hypocontinuous in its second argument
(see \cite[II.5.2]{SaW}) and hence sequentially continuous
(see \cite[p.\,157, Remark following \S40, 1., (5)]{Koe}).\\[2.4mm]
Fix $\gamma\in C^\infty_c(G)$ and let~$K$ be its support.
Let $\Phi\colon E^\infty\to C^\infty(G,E)$ be as in~(\ref{defPhi}).
The map $\Pi(\gamma,.)$ will be continuous
if we can show that $h:=\Phi\circ \Pi(\gamma,.):$\linebreak
$E\to C^\infty(G,E)$ is continuous.
By Lemma~\ref{explaw}, this will hold if $\wh{h}\colon E\times G\to E$,
\begin{equation}\label{stasta}
(v,g)\mto \pi(g)\!\!\int_G\gamma(x)\pi(x,v)\,d\lambda_G(x)\!=\!
\!\int_G\gamma(g^{-1}y)\pi(y,v)\,d\lambda_G(y)
\end{equation}
is $C^{0,\infty}$. It suffices to show that $\wh{h}$ is $C^\infty$.
Given $g_0\in G$, let $U\sub G$ be a relatively compact, open
neighbourhood of $g_0$. 
We show that $\wh{h}$ is smooth on $E\times U$.
For $g\in U$,
the domain of integration can be replaced by the compact set $\wb{U}K\sub G$,
without changing the second integral in~(\ref{stasta}).
By Lemma~\ref{specsit},
\[
(E\times G)\times G \to E\,,\quad ((v,g),y)\mto \gamma(g^{-1}y)\pi(y,v)
\]
is~$C^{\infty,0}$.
Its restriction to $(E\times U)\times \wb{U}K$ therefore
satisfies the hypotheses of Lemma~\ref{pardep}, and hence
the parameter-dependent integral $\wh{h}|_{E\times U}$
is smooth.\\[2.4mm]
Next, fix $v\in E$.
For $\gamma\in C^\infty_c(G)$, define
$\psi(\gamma)\colon G\to C^0_c(G,E)$
via $\psi(\gamma)(g)(y)$\linebreak
$:=\gamma(g^{-1}y)\pi(y,v)$.
We claim:
\begin{itemize}
\item[(a)]
\emph{$\psi(\gamma)\in C^\infty(G,C^0_c(G,E))$ for each $\gamma\in C^\infty_c(G)$; and}
\item[(b)]
\emph{The linear map
$\psi\colon C^\infty_c(G)\to C^\infty(G, C^0_c(G,E))$
is continuous.}
\end{itemize}
Note that the integration operator
$I\colon C^0_c(G,E)\to E$, $\eta\mto \int_G\eta(y)\,d\lambda_G(y)$
is continuous linear,\footnote{In fact, the restriction of $I$ to $C^0_K(G,E)$
is continuous for each compact set $K\sub G$,
because $q(I(\gamma))\leq \|\gamma\|_{q,\infty}\lambda_G(K)$
for each continuous seminorm~$q$ on~$E$ and $\gamma\in C^0_K(G,E)$.}
entailing that also
\[
C^\infty(G,I)\colon C^\infty(G,C^0_c(G,E))\to C^\infty(G,E),\quad
f\mto I\circ f
\]
is continuous linear (see \cite{GaN} or \cite[Lemma~4.13]{ZOO}).
If the claim holds, then the formula $\Phi\circ \Pi(.,v)=C^\infty(G,I)\circ \psi$
shows that $\Phi\circ \Pi(.,v)$ is continuous, and thus $\Pi(.,v)$
is continuous.
Hence, it only remains to establish the claim.\\[2.4mm]
To prove (a), fix $\gamma\in C^\infty_c(G)$
and let $K$ be its support.
It suffices to show that, for each $g_0\in G$
and relatively compact, open neighbourhood~$U$ of $g_0$ in~$G$,
the restriction
$\psi(\gamma)|_U$ is smooth.
As the latter has its image in $C^0_{\wb{U}K}(G,E)$,
which is a closed vector subspace of both $C^0_c(G,E)$ and $C^0(G,E)$
with the same induced topology,
it suffices to show that $h:=\psi(\gamma)|_U$ is smooth as
a map to $C^0(G,E)$ \cite[Lemma~10.1]{BGN}.
But this is the case (by Lemma~\ref{explaw}), as
\[
\wh{h}\colon U\times G\to E,\quad (g,y)\mto \gamma(g^{-1}y)\pi(y,v)
\]
is $C^{\infty,0}$ (by Lemma~\ref{specsit}).
By Lemma~\ref{explaw},
to prove~(b) we need to check that
\[
\wh{\psi}\colon C^\infty_c(G)\times G\to C^0_c(G,E),\quad \wh{\psi}(\gamma,g)(y)= \gamma(g^{-1}y)\pi(y,v)
\]
is $C^{0,\infty}$. We show that $\wh{\psi}$ is $C^\infty$.
By Lemma~\ref{lemB2}, the map
\[
\theta \colon C^\infty_c(G)\to C^0_c(G,E),\quad \theta (\gamma)(y):=\gamma(y)\pi(y,v)
\]
is continuous linear.
The map $\tau\colon G\times C^\infty_c(G)\to C^\infty_c(G)$, $\tau(g,\gamma)(x)=\gamma(g^{-1}x)$
is smooth, by Lemma~\ref{lemB1}.
Since $\wh{\psi}(\gamma,g)=\theta(\tau(g,\gamma))$, also $\wh{\psi}$
is smooth.
\end{proof}
{\bf Proof of Proposition~B.}
As the inclusion map $E^\infty\to E$ is continuous, the final assertions follow
once we have continuity of $\Pi\colon C^\infty_c(G)\times E\to E^\infty$.\\[2.7mm]
We first assume that~$E$ is a Banach space.
By Lemma~\ref{hypoetc}, $\Pi$ is hypocontinuous
in the second argument.
As the unit ball $B\sub E$ is bounded,
it follows that $\Pi|_{C^\infty_c(G)\times B}$ is continuous
(Proposition~4 in [10, Chapter III, \S5, no. 3]).
Since~$B$ is a $0$-neighbourhood,
we conclude that~$\Pi$ is continuous.\\[2.4mm]
If $E$ is a proto-Banach $G$-module,
then the topology on $E$ is initial with respect to
a family $f_j\colon E\to E_j$ of continuous linear
$G$-equivariant maps to certain Banach $G$-modules $(E_j,\pi_j)$.
As a consequence, the topology on $C^\infty(G,E)$
is initial with respect to the mappings
\[
h_j:=C^\infty(G,f_j)\colon C^\infty(G,E)\to C^\infty(G,E_j),\quad
\gamma\mto f_j\circ \gamma
\]
(see \cite{GaN}).
Therefore, the topology on $E^\infty$
is initial with respect to the maps $h_j\circ \Phi$
(with $\Phi$ as in (\ref{defPhi})).
Now consider $\Phi_j\colon E_j^\infty\to C^\infty(G,E_j)$, $w\mto (\pi_j)_w$.
Since $f_j\circ \pi_v=(\pi_j)_{f_j(v)}$,
we have $f_j(E^\infty)\sub (E_j)^\infty$.
Moreover, the topology on~$E^\infty$
is initial with respect to the maps
$h_j\circ \Phi=\Phi_j\circ f_j$.
By the Banach case already discussed,
$\Pi_j\colon C^\infty_c(G)\times E_j\to (E_j)^\infty$, $\Pi_j(\gamma,w):=\int_G\gamma(y)\pi_j(y,w)\,d\lambda_G(y)$
is continuous for each $j\in J$. Since
$\Phi_j\circ f_j\circ \Pi=\Phi_j\circ \Pi_j\circ (\id_{C^\infty_c(G)}\times f_j)$
is continuous for each~$j$, so is $\Pi$.\,\Punkt
\section{Preliminaries for Propositions C, D and E}\label{preC}
If $E$ is a vector space and $(U_j)_{j\in J}$ a family of subsets $U_j\sub E$, we abbreviate
\[
\sum_{j\in J}U_j:=\bigcup_F\,\sum_{j\in F}U_j,
\]
for $F$ ranging through the set of finite subsets of~$J$.
\begin{numba}
If $E$ and $F$ are complex locally convex spaces,
then a function $f\colon U\to F$ on an open set $U\sub E$
is called \emph{complex analytic} (or $\C$-analytic)
if $f$ is continuous and each $x\in U$ has a neighbourhood
$Y\sub U$ such that
\begin{equation}\label{expa}
(\forall y\in Y)\;\;f(y)=\sum_{n=0}^\infty p_n(y-x)
\end{equation}
pointwise,
for some continuous homogeneous complex polynomials $p_n\colon \!E\!\to\! F$
of degree~$n$ (see \cite{BaS}, \cite{RES}, \cite{GaN}, \cite{Mil}
for further information).
\end{numba}
\begin{numba}\label{defnra}
If $E$ and $F$ are real locally convex spaces,
following \cite{Mil}, \cite{RES} and \cite{GaN}
we call a function $f\colon U\to F$ on an open set $U\sub E$
\emph{real analytic} (or $\R$-analytic)
if it extends to a $\C$-analytic function
$\wt{U}\to F_\C$ on an open set $\wt{U}\sub E_\C$.
\end{numba}
\begin{numba}\label{numrem}
Both concepts are chosen in such a way that compositions
of $\K$-analytic maps are $\K$-analytic (for $\K\in\{\R,\C\}$).
They therefore give rise to notions of $\K$-analytic manifolds modelled
on locally convex spaces and $\K$-analytic mappings between them.
If $E$ is finite-dimensional (or Fr\'{e}chet) and $F$ is sequentially complete
(or Mackey-complete),\footnote{In the sense
that each Mackey-Cauchy sequence in~$F$ converges (see \cite{KaM}).}
then a map $f\colon E\supseteq U\to F$ as in \ref{defnra}
is $\R$-analytic if and only if it is continuous and admits
local expansions (\ref{expa}) into continuous homogeneous real polynomials
(cf.\ \cite[Theorem~7.1]{BaS} and \cite[Lemma~1.1]{GN2}),
i.e., if and only if it is real analytic in the sense of~\cite{BaS}.
\end{numba}
By the next lemma (proved in Appendix~A, like all other lemmas from this section),
our notion of analytic vector coincides with that in~\cite{GKS}.
\begin{la}\label{extGC}
Let $G$ be a connected Lie group with $G\sub G_\C$,
$\pi\colon G\times E\to E$ be a topological $G$-module
and $v\in E$. Then $v\in E^\omega$
if and only if the orbit map $\pi_v$
admits a $\C$-analytic extension $GV\to E$
for some open identity neighbourhood $V\sub G_\C$.
\end{la}
\begin{numba}\label{collect}
The map $d\colon G\to [0,\infty[$ from (\ref{defnd}) has the following elementary properties:
\begin{equation}\label{firstd}
(\forall x,y\in G)\;\; d(xy)\leq d(x)+d(y)\quad\mbox{and}\quad d(x^{-1})=d(x)\,.
\end{equation}
It is essential for us that
\begin{equation}\label{secd}
\int_Ge^{-\ell d(g)}\,d\lambda_G(g)\;<\; \infty
\end{equation}
for some $\ell\in \N_0$,
by \cite[\S1, Lemme~2]{Gar}.
For each $G$-continuous seminorm $p$ on a topological $G$-module $\pi\colon G\times E\to E$,
there exist $C,c>0$ such that
\begin{equation}\label{good2}
(\forall g\in G)(\forall v\in E)\quad p(\pi(g,v))\leq Ce^{cd(g)}p(v),
\end{equation}
as a consequence of \cite[\S2, Lemme~1]{Gar}.
\end{numba}
\begin{numba}
Given a connected Lie group~$G$ with $G\sub G_\C$,
let $\wt{A}_n(G)$ be the space of all $\C$-analytic functions
$\eta\colon V_nG\to\C$ such that
\begin{equation}\label{requir}
\|\eta\|_{K,N}:=\sup\{|\eta(z^{-1}g)|e^{Nd(g)}\colon z\in K,g\in G\}\;<\;\infty
\end{equation}
for each $N\in \N_0$ and compact set $K\sub V_n$ (for $V_n$ as in the introduction).
Make $\wt{A}_n(G)$ a locally convex space using the norms $\|.\|_{K,N}$.
It is essential for us that the map
\[
\wt{\cA}_n(G)\to \cA_n(G),\quad \eta\mto \eta|_G
\]
is an isomorphism of topological vector spaces.
It inverse is the map $\gamma\mto \wt{\gamma}$
taking $\gamma$ to its unique $\C$-analytic extension
$\wt{\gamma}\colon V_nG\to \C$
(see \cite[Lemma~4.3]{GKS}).
Given $\gamma\in \cA_n(G)$ and $K,N$ as before,
we abbreviate $\|\gamma\|_{K,N}:=\|\wt{\gamma}\|_{K,N}$.
\end{numba}
The next two lemmas show that the space
$\wt{\cA}_n(G)$ and its topology remain unchanged if,
instead, one requires (\ref{requir}) for
all compact subsets $K\sub GV_n$.
\begin{la}\label{betterK}
If $K\sub GV_n$ is compact,
then there exists a compact set $L\sub V_n$
such that $GK\sub GL$.
\end{la}
\begin{la}\label{changeK}
If $K,L\sub GV_n$ are compact sets such that
$GK\sub GL$,
let $\theta:=\max\{d(h)\colon h\in KL^{-1}\}<\infty$.
Then
$\|\gamma\|_{K,N}\leq e^{N\theta}\|\gamma\|_{L,N}$,
for all $\gamma\in \cA_n(G)$ and $N\in \N_0$.
\end{la}
We set up a notation for seminorms on $E_n$ which define its topology:
\begin{numba}\label{newnota}
Let $G$ be a connected Lie group with $G\sub G_\C$
and $E$ a topological $G$-module.
If $K\sub GV_n$ is compact and $p$ a continuous seminorm
on $E$, set
\begin{equation}\label{ncesems}
\|v\|_{K,p}:=\sup\{p(\wt{\pi}_v(z))\colon z\in K\}\quad\mbox{for $v\in E_n$.}
\end{equation}
\end{numba}
We need a variant of Lemma~\ref{pardep} ensuring complex
analyticity.
The $C^{1,0}_\C$-maps encountered here are defined as in \ref{RS},
except that complex directional derivatives are used
in the first factor.
\begin{la}\label{cxpardep}
Let $Z$, $E$ be complex locally convex spaces, $U\sub Z$ be open,
$Y$ a $\sigma$-compact locally compact space, $\mu$ a Borel measure on~$Y$
which is finite on compact sets,
and $f\colon U\times Y \to E$ be a $C^{1,0}_\C$-map.
Assume that~$E$ is sequentially complete
and assume that, for each continuous seminorm~$q$
on~$E$, there exists a $\mu$-integrable function $m_q \colon Y\to [0,\infty]$
such that $q(f(z,y))\leq m_q(y)$ for all $(z,y)\in U\times Y$.
Then $g(z):=\int_Y f(z,y)\,d\mu(y)$
exists in~$E$ as an absolutely convergent integral,
for all $z\in U$, and $g\colon U\to E$ is
$\C$-analytic.
\end{la}
Also the following fact from \cite{GKS} will be used:
\begin{la}\label{intrep1}
Let $G$ be a connected Lie group with $G\sub G_\C$ and
$(E,\pi)$ be a sequentially complete proto-Banach $G$-module.
Let $n\in \N$. Then $w:=\Pi(\gamma,v)\in E_n$ for all $\gamma\in \cA_n(G)$
and $v\in E$. The $\C$-analytic extension of the orbit map $\pi_w$ of $w$ is given by
\begin{equation}\label{integals}
\wt{\pi}_w\colon GV_n\to E,\quad z\mto\int_G\wt{\gamma}(z^{-1}y)\pi(y,v)\,d\lambda_G(y).
\end{equation}
The $E$-valued integrals in {\rm(\ref{integals})} converge absolutely.
\end{la}
The next two lemmas will enable a proof of Proposition~D.
\begin{la}\label{sin}
Let $G$ be a connected Lie group such that $G\sub G_\C$
and $L(G)$ is a compact Lie algebra. Then there exists a basis
$(V_n)_{n\in\N}$ of open, connected, relatively compact
identity neighbourhoods $V_n\sub G_\C$ such that
$\wb{V_{n+1}}\sub V_n$ and
$gV_ng^{-1}=V_n$ for all $n\in\N$ and $g\in G$.
In addition, one can achieve that
$\{gxg^{-1}\colon g\in G,\,x\in K\}$ has compact
closure in $V_n$,
for each $n\in \N$ and each compact subset $K\sub V_n$.
\end{la}
\begin{la}\label{intrnew}
Let $G$ be a connected Lie group with a compact Lie algebra and $G\sub G_\C$.
If $(E,\pi)$ is a sequentially complete proto-Banach $G$-module,
let $(V_n)_{n\in\N}$ be as in Lemma~{\rm\ref{sin}}
and define $E_n$ using $V_n$, for each $n\in\N$.
Then $w:=\Pi(\gamma,v)\in E_n$ for all $\gamma\in \cA(G)$
and $v\in E_n$. The $\C$-analytic extension of the orbit map $\pi_w$ of $w$ is given by
\begin{equation}\label{integnew}
\wt{\pi}_w\colon GV_n\to E,\quad z\mto\int_G\gamma(y)\wt{\pi}_v(zy)\,d\lambda_G(y).
\end{equation}
The $E$-valued integrals in {\rm(\ref{integnew})} converge absolutely.
\end{la}
\begin{la}\label{reach}
In Lemma~{\rm\ref{intrep1}}, the bilinear map
$\Pi\colon \cA(G)\times E\to E^\omega$ is\linebreak
separately continuous,
hypocontinuous in the second argument and\linebreak
sequentially continuous.
If $E$ is barrelled $($e.g., if $E$ is a Fr\'{e}chet space$)$,
then $\Pi$ is hypocontinuous in both arguments.
\end{la}
Recall that a topological space $X$ is said to be sequentially compact if it is Hausdorff and every sequence in $X$ has a convergent subsequence \cite[p.\,208]{Eng}.
\begin{la}\label{boundd}
If $E$ is a locally convex space and $K\sub E$ a sequentially\linebreak
compact subset, then $K$ is bounded in~$E$.
\end{la}
The following fact has also been used
in \cite[Appendix~B]{GKS} (without proof).
\begin{la}\label{Montel}
$\cA_n(G)$ is a Montel space,
for each Lie group $G$ such that $G\sub G_\C$
and $n\in \N$.
\end{la}
\section{Proof of Proposition~C}
Let $W$ be a $0$-neighbourhood in $E^\omega$.
Then there are $0$-neighbourhoods $S_n\sub E_n$ for $n\in \N$
such that $\sum_{n\in \N}S_n\sub W$.
Shrinking $S_n$ if necessary, we may assume that $S_n=\{ v \in E_n\colon \|v\|_{K_n,p_n}<1\}$
for some compact subset $K_n\sub G V_n$ and
$G$-continuous seminorm~$p_n$ on~$E$
(with notation as in \ref{newnota}).\\[2.4mm]
For the intermediate steps of the proof,
we can proceed similarly as in \cite[proof of Proposition~4.6]{GKS}:
By \ref{collect},
there exist $C_n>0$, $m_n\in \N_0$ such that
$p_n(\pi(g,v))\leq p_n(v)C_ne^{m_n d(g)}$
for all $g\in G$ and $v\in E$.
Pick $\ell\in \N_0$ with $C:=\int_Ge^{-\ell d(y)}\,d\lambda_G(y)<\infty$
(see (\ref{secd})) and set $N_n:=m_n+\ell$.
For $\gamma\in \cA_n(G)$ and $v\in E$,
we have $w:=\Pi(\gamma,v)\in E_n$ by Lemma~\ref{intrep1},
and $\wt{\pi}_w$ is given by (\ref{integals}).
The integrand in (\ref{integals}) admits the estimate
$p_n(\wt{\gamma}(z^{-1}y) \pi(y,v)) \leq p_n(v)C_n\|\gamma\|_{K_n,N_n}e^{-\ell d(y)}$
for all $z\in K_n$, $y\in G$ (cf.\ (\ref{dochuse}) with $x=1$). Hence
\[
p_n(\wt{\pi}_w(z)) \leq p_n(v)C C_n\|\gamma\|_{K_n,N_n}.
\]
By Lemma~\ref{betterK},
there exists a compact subset $L_n\sub V_n$ such that
$GK_n\sub GL_n$.
Let $\theta_n:=\max\{d(h)\colon h\in K_nL_n^{-1}\}$.
If $E$ has the countable neighbourhood property,
then there exists a continuous seminorm $p$ on~$E$
and constants $a_n\geq 0$ such that
$p_n\leq a_np$ for all $n\in \N$.
Thus, using Lemma~\ref{changeK},
\[
p_n(\wt{\pi}_w(z))\leq a_np(v)C C_n e^{\theta_n N_n}\|\gamma\|_{L_n,N_n}\,.
\]
Choose $\ve_n>0$ so small that $\ve_na_n CC_n e^{\theta_n N_n}<1$,
and set $P_n:=\{\gamma\in \cA_n(G)\colon\|\gamma\|_{L_n,N_n}<\ve_n\}$.
Then $\|\Pi(\gamma,v)\|_{K_n,p_n}\leq \ve_na_n C C_n e^{\theta_n N_n}<1$
for all $v\in B^p_1(0)$ and $\gamma\in P_n$
and thus $\Pi(P_n\times B^p_1(0))\sub S_n$.
Then $P:=\sum_{n\in \N}P_n$ is a $0$-neighbourhood in $\cA(G)$ and
\[
\Pi(P\times B^p_1(0)))\sub \sum_{n\in \N}  \Pi(P_n\times B^p_1(0)))\sub \sum_{n\in \N} S_n\sub W\,.
\]
Hence $\Pi$ is continuous at $(0,0)$ and hence continuous
(as $\Pi$ is bilinear).
\section{Proof of Proposition~D}\label{secpropD}
Choose $(V_n)_{n\in\N}$ as in Lemma~\ref{sin}.
Let $W$ be a $0$-neighbourhood in $E^\omega$.
Then there are $0$-neighbourhoods $S_n\sub E_n$ for $n\in \N$
such that $\sum_{n\in \N}S_n\sub W$.
Shrinking $S_n$ if necessary, we may assume that $S_n=\{ v \in E_n\colon \|v\|_{K_n,p_n}<1\}$
for some compact subset $K_n\sub G V_n$ and
$G$-continuous seminorm~$p_n$ on~$E$
(with notation as in \ref{newnota}).
After increasing $K_n$, we may assume that $K_n=A_nB_n$ with compact
subsets $A_n\sub G$ and $B_n\sub V_n$.\\[2.4mm]
By \ref{collect},
there exist $C_n>0$, $m_n\in \N_0$ such that
$p_n(\pi(g,v))\leq p_n(v)C_ne^{m_n d(g)}$
for all $g\in G$ and $v\in E$. Then
$R_n:=\sup\{e^{m_nd(x)}\colon x\in A_n\}<\infty$.
Pick $\ell\in \N_0$ with $C:=\int_Ge^{-\ell d(y)}\,d\lambda_G(y)<\infty$
(see (\ref{secd})).\\[2.3mm]
For $i\in \N$, let $N_i$
be the maximum of $m_1+\ell,\ldots, m_i+\ell$.
%Define $r_i:=
%\sup\{e^{N_i d(x)}\colon x\in A_i\}<\infty$.
Pick
$\ve_i\in \,]0,2^{-i}[$ so small
that
$R_i C_iC \ve_i<2^{-i}$.
Set $P_i:=\{\gamma\in \cA_i(G)\colon
\|\gamma\|_{B_i,N_i}<\ve_i\}$.
Then $P:=\sum_{i\in \N}P_i$ is a $0$-neighbourhood in $\cA(G)$.\\[2.3mm]
For $j\in \N$,
let $q_j$ be the pointwise maximum of $p_1,\ldots,p_j$.
Let $H_j$ be the closure of $\{gyg^{-1}\colon g\in G,\, y\in B_j\}$
in $V_j$. Choose $\delta_j\in\,]0,2^{-j}[ $ so small
that $CC_jR_j\delta_j<2^{-j}$. Set $Q_j:=\{v\in E_j\colon \|v\|_{H_j,q_j}<\delta_j\}$.
Then $Q:=\sum_{j\in \N}Q_j$ is a $0$-neighbourhood in~$E^\omega$.\\[2.3mm]
We now verify that $\Pi(P\times Q)\sub W$,
entailing that the bilinear map $\Pi$ is continuous at $(0,0)$ and thus continuous.
To this end, let $\gamma\in P$, $v\in Q$.
Then $\gamma=\sum_{i=1}^\infty \gamma_i$ and
$v=\sum_{j=1}^\infty v_j$
with suitable
$\gamma_i\in P_i$
and $v_j\in Q_j$,
such that $\gamma_i\not=0$
for only finitely many $i$
and $v_j\not=0$
for only finitely many~$j$.
Abbreviate $w_{i,j}:=\Pi(\gamma_i,v_j)$.\\[2.3mm]
If $j<i$, then $w_{i,j}\in E_j$ by Lemma~\ref{intrnew}.
Moreover, (\ref{dochunew}) shows that
\begin{equation}\label{ca1ok}
\|w_{i,j}\|_{K_j,p_j}
\leq
C C_j R_j \|\gamma_i\|_{m_j+\ell}\|v_j\|_{H_j,p_j}
<C C_j R_j\ve_i\delta_j<2^{-i}2^{-j}.
\end{equation}
If $i\leq j$, then $w_{i,j}\in E_i$ by Lemma~\ref{intrep1},
and (\ref{dochuse}) implies that
\begin{equation}\label{ca2ok}
\|w_{i,j}\|_{K_i,p_i}\leq
R_i C_iC\|\gamma_i\|_{B_i,m_i+\ell}\,p_i(v_j)
\leq R_i C_iC \ve_i\delta_j<2^{-i}2^{-j}.
\end{equation}
For each $n\in \N$, we have $\sum_{\min\{i,j\}=n}w_{i,j}\in S_n$,
since (by (\ref{ca1ok}) and (\ref{ca2ok}))
\[
\left\|\sum_{\min\{i,j\}=n}w_{i,j}\right\|_{K_n,p_n}\leq \sum_{\min\{i,j\}=n}2^{-i}2^{-j} <1\,.
\]
Hence $\Pi(\gamma, v)
=\sum_{n=1}^\infty\sum_{\min\{i,j\}=n}w_{i,j}\in \sum_{n\in \N} S_n\sub W$, as required.
\section{Proof of Proposition~E}\label{secpropE}
We use a variant of \cite[Theorem~4]{Yam},
which does not require that the direct sequence is strict.
\begin{prop}\label{yamprop}
Let $G_1\sub G_2\sub\cdots$ be a sequence of metrizable
topological groups such that all inclusion maps $G_n\to G_{n+1}$ are continuous
homomorphisms. Let~$\cO_{DL}$ be the direct
limit topology
on $G:=\bigcup_{n\in \N}G_n$ and $\cO_{TG}$ the topology
making $G$ the direct limit
$\dl\, G_n$\vspace{-.3mm} as a topological group.
Assume:
\begin{itemize}
\item[\rm(a)]
For each $n\in \N$,
there is $m>n$ such that the set $G_n$ is not open in~$G_m$;
\item[\rm(b)]
There exists $n\in \N$ such that, for all identity neighbourhoods $U\sub G_n$
and $m>n$,
the closure of $U$ in $G_m$ is not compact; and:
\item[\rm(c)]
There exists a Hausdorff topology $\cT$ on~$G$ making
each inclusion map $G_n\to G$ continuous,
and such that
every sequentially compact subset
of $(G,\cT)$
is contained in some $G_n$ and compact in there.
\end{itemize}
Then $\cO_{DL}$ does not make the group multiplication
$G\times G\to G$ continuous and hence $\cO_{DL}\not=\cO_{TG}$.
\end{prop}
\begin{rem}
By definition, a set $M\sub G$ is open (resp., closed)
in $(G,\cO_{DL})$ if and only if $M\cap G_n$ is open
(resp., closed) in $G_n$ for each $n\in \N$.
By contrast, $\cO_{TG}$ is defined as the finest
among the topologies on~$G$ making
$G$ a topological group and each inclusion map $G_n\to G$
continuous. 
See \cite{JFA}, \cite{GHK}, \cite{Hir}, \cite{Yam}
for comparative discussions of $\cO_{DL}$
and $\cO_{TG}$.
\end{rem}
{\bf Proof of Proposition~\ref{yamprop}.}
If $G_n$ is not open in~$G_m$ with $m>n$,
then $G_n$ also fails to be open
in $G_k$ for all $k>m$. In fact, let $i_{m,k}\colon G_m\to G_k$
be the continuous inclusion map. If $G_n$ was open in~$G_k$,
then $G_n= i_{m,k}^{-1}(G_n)$ would also be open in~$G_m$,
contradiction.
Similarly, if $n$ is as in (b) and $k>n$, then also $G_k$
does not have an identity neighbourhood which has compact closure in
$G_\ell$ for some $\ell>k$. In fact, if $U$ was such a neighbourhood,
then $i_{k,n}^{-1}(U)$ would be an identity neighbourhood in~$G_n$
whose closure in $G_\ell$ is contained in $\wb{U}$ and hence compact,
contradiction.
After passing to a subsequence, we may hence assume that
$G_n$ is not open in $G_{n+1}$ (and hence not an identity neighbourhood),
for each $n\in \N$.
And we can assume that, for each $n\in \N$
and identity neighbourhood $U\sub G_n$,
for each $m>n$ the closure of $U$ in~$G_m$ is not compact.\\[2.5mm]
If $\cO_{DL}$ makes the group multiplication continuous,
then for every identity neighbourhood $U\sub (G,\cO_{DL})$,
there exists an identity neighbourhood $W\sub (G,\cO_{DL})$
such that $WW\sub U$. Then
\begin{equation}\label{letfail}
(\forall n\in \N) \quad (W\cap G_1)(W\cap G_n) \sub U\cap G_n.
\end{equation}
Thus, assuming (a)--(c), $\cO_{DL}$ will not be a group topology
if we can construct an identity neighbourhood $U\sub (G,\cO_{DL})$
such that (\ref{letfail}) fails for each~$W$.\\[2.4mm]
To achieve this, let
$d_n$ be a metric on~$G_n$ defining its topology, for $n\in \N$.
Let $V_1\supseteq V_2\supseteq \cdots$ be a basis
of identity neighbourhoods in~$G_1$.\\[2.4mm]
Since $G_n$ is metrizable and $G_{n-1}$ not an identity neighbourhood in~$G_n$,
for each $n\geq 2$ we find a sequence $(x_{n,k})_{k\in \N}$
in $G_n\setminus G_{n-1}$ such that $x_{n,k}\to 1$ in~$G_n$ as $k\to\infty$.
Let $K:=\wb{V_{n-1}}$ be the closure of $V_{n-1}$ in~$G_n$.
Then~$K$ cannot be sequentially compact in~$G$,
as otherwise~$K$ would be compact in~$G_m$ for some
$m\in \N$ (by (c)), contradicting~(b).
Hence $K$ contains a sequence $(w_{n,k})_{k\in \N}$
which does not have a convergent subsequence in~$G$,
and hence does not have a convergent subsequence in~$G_m$
for any $m\geq n$. Pick $z_{n,k}\in V_{n-1}$ such that
\begin{equation}\label{metrapp}
d_n(w_{n,k},z_{n,k})<\frac{1}{k}.
\end{equation}
Then also $(z_{n,k})_{k\in \N}$
does not have a convergent subsequence in $G_m$ for any $m\geq n$
(if $z_{n,k_\ell}$ was convergent,
then $w_{n,k_\ell}$ would converge to the same limit, by (\ref{metrapp})).
Moreover, $(z_{n,k}x_{n,k})_{k\in \N}$
does not have a convergent subsequence $(z_{n,k_\ell}x_{n,k_\ell})_{\ell\in\N}$
in $G_m$ for any $m\geq n$ (because then $z_{n,k_\ell}=(z_{n,k_\ell}x_{n,k_\ell})x_{n,k_\ell}^{-1}$
would converge (contradiction).\\[2.4mm]
As a consequence, the set $C_n:=\{z_{n,k}x_{n,k}\colon k\in \N\}$
is closed in~$G_m$ for each $m\geq n$.
Also note that $z_{n,k}x_{n,k}\in G_n\setminus G_{n-1}$
and thus $C_n\sub G_n\setminus G_{n-1}$.
Hence $A_n:=\bigcup_{\nu=2}^n C_\nu$ is a closed subset of~$G_n$ for each $n\geq 2$,
and $A:=\bigcup_{n\geq 2}A_n$ is closed in $(G,\cO_{DL})$
because $A\cap G_n=A_n$ is closed for each $n\geq 2$.
Thus $U:=G\setminus A$ is open in $(G,\cO_{DL})$,
and $U\cap G_n=G_n\setminus A_n$.
We show that $WW\not\sub U$ for any $0$-neighbourhood
$W\sub G$. In fact, there is $n\geq 2$ such that $V_{n-1}\sub W\cap G_1$.
Since $x_{n,k}\to 0$ in $G_n$ as $k\to\infty$, there is $k_0\in \N$ such that
$x_{n,k}\in W\cap G_n$ for all $k\geq k_0$.
Also, $z_{n,k_0}\in V_{n-1}$. Hence
$z_{n,k_0}x_{n,k_0}\in (W\cap G_1)(W\cap G_n)$.
But $z_{n,k_0}x_{n,k_0}\in A_n$ and thus $z_{n,k_0}x_{n,k_0}\not\in U\cap G_n$.
As a consequence, $WW\not\sub U$.\,\vspace{2.4mm}\Punkt

\noindent
Because the locally convex direct limit topology $\cO_{lcx}$
on an ascending union of locally convex spaces
coincides with $\cO_{TG}$ \cite[Proposition~3.1]{Hir}, we obtain:
\begin{cor}\label{yamcor}
Let $E_1\sub E_2\sub\cdots$ be metrizable
locally convex spaces such that all inclusion maps $E_n\to E_{n+1}$ are continuous
linear. On $E:=\bigcup_{n\in \N}E_n$, let $\cO_{DL}$ be the direct limit topology
and $\cO_{lcx}$ be the locally convex direct limit topology.
Then $\cO_{DL}\not=\cO_{lcx}$ if {\rm(a)--(c)} are satisfied:
\begin{itemize}
\item[\rm(a)]
For each $n\in \N$,
there exists $m>n$ such that $E_m\setminus E_n\not=\emptyset$;
\item[\rm(b)]
There exists $n\in \N$ such that, for each $0$-neighbourhood $U\sub E_n$
and $m>n$, the closure of $U$ in $E_m$ is not compact;
\item[\rm(c)]
$\cO_{lcx}$ is Hausdorff and every sequentially compact subset of $(E,\cO_{lcx})$
is contained in some~$E_n$ and compact in there.\,\Punkt
\end{itemize}
\end{cor}
It is convenient to make the special choice of the $V_n$ proposed in~\cite{GKS}
now. To this end, extend ${\bf g}$ to a left-invariant Riemannian metric
on $G_\C$, write\linebreak
${\bf d}_\C\colon G_\C\times G_\C\to[0,\infty[$
for the associated distance function,
and set $d_\C(z):={\bf d}_\C(z,1)$ for $z\in G_\C$.
For $\rho>0$, let
\[
B_\rho:=\{z\in G_\C\colon d_\C(z)<\rho\}
\]
be the respective open ball around~$1$. Then the sets
$V_n:=B_{1/n}$, for $n\in \N$ have properties as described in the introduction.
Notably, $\wb{B_\rho}$ is compact for each $\rho>0$,
and hence also each $\wb{V_n}$ (see \cite[p.\,74]{Gar}).
\begin{la}\label{notstat}
Let $G$ be a connected Lie group with $G\sub G_\C$
and $G\not=\{1\}$.
Then the sequence $\cA_1(G)\sub \cA_2(G)\sub\cdots$
does not become stationary.
\end{la}
\begin{proof}
Step~1. If $G$ is compact, then
$G$ is isomorphic to a closed subgroup of some
unitary group.
Hence~$G$ can be realized as a closed $\R$-analytic submanifold of some~$\R^k$
(which is also clear from \cite[Theorem 3]{Gra}),
entailing that $\R$-analytic functions (like restrictions of linear functionals)
separate points on~$G$. In particular, there exists a
non-constant $\R$-analytic function $\gamma\colon G\to\R$,
and the latter then extends to a $\C$-analytic function $\wt{\gamma}$
on some neighbourhood of $G$ in $G_\C$,
which (since $G$ is compact) can be assumed of the form $GV_m$ for some $m\in \N$.
Then $\wt{\gamma}\in \wt{A}_m(G)$
(noting that $d$ is bounded).\\[2.4mm]
Step 2. If $G$ is not compact,
we recall the ``regularized distance function'':
There exists $m\in \N$ and a $\C$-analytic
function $\wt{d}\colon GV_m\to \C$
such that
\[
C:=\sup\{|\wt{d}(gz)-d(g)|\colon g\in G, z\in V_m\}\;<\;\infty
\]
(see \cite[Lemma~4.3]{Kro}).
Then also $\theta\colon V_mG\to \C$, $\theta(z):=\wt{d}(z^{-1})$
is $\C$-analytic, and $|\theta(zg)-d(g)|=|\wt{d}(g^{-1}z^{-1})-d(g^{-1})|\leq C$
for all $z\in V_m$ and $g\in G$.
Since $\{x\in G\colon d(g)\leq R\}$ is compact for each $R>0$
(see \cite{Gar}), for each $R>0$ there exists $g\in G$
such that $d(g)>R$ and thus $|\theta(g)|>R-C$.
Hence $\theta$ is not constant,
and hence also $\theta^2$ and $\Repart(\theta^2)$
are not constant.
If $N\in \N_0$, there is $r_N>0$ such that
\[
a^2-2aC-C^2\geq Na\quad\mbox{for all $\,a\geq r_N$.}
\]
Since $\theta(zg)^2=d(g)^2+2(\theta(zg)-d(g))d(g)+(\theta(zg)-d(g))^2$,
we deduce that
\[
\Repart(\theta(zg)^2)\geq
d(g)^2-2Cd(g)-C^2\geq Nd(g)
\]
for all $z\in V_m$ and all $g\in G$ such that $d(g)\geq r_N$.
We also have
\[
\Repart(\theta(zg)^2)\geq -|\theta(zg)^2|\geq -(d(g)+C)^2\quad\mbox{for all
$z\in V_m$, $g\in G$.}
\]
Thus $\gamma\colon V_mG\to\C$, $\gamma(z):=e^{-\theta(z)^2}$
is non-constant, $\C$-analytic and
\[
|\gamma(zg)|e^{Nd(g)}
=e^{-\Repart(\theta(zg)^2)}e^{Nd(g)}\; \leq \; e^{(r_N+C)^2+Nr_N}
\]
for all $z\in V_m$, $g\in G$.
Hence $\|\gamma\|_{K,N}<\infty$ for each compact set $K\sub V_m$
and $N\in \N_0$. Thus $\gamma\in \wt{A}_m(G)$,
and hence $\wt{A}_n(G)\not=\{0\}$ for all $n\geq m$.\\[2.4mm]
Step~3. In either case, let $\wt{\gamma}\in \wt{\cA}_m(G)$ be a non-constant
function, and $n>m$. Then also $\gamma:=\wt{\gamma}|_G$ is non-constant (as $G$ is totally real in~$G_\C$).
If $G$ is compact, then $|\gamma|$ attains a maximum $a>0$.
If $G$ is non-compact, then $\gamma$ vanishes at infinity.
Hence $|\gamma(G)|\cup\{0\}$ is compact and hence $|\gamma|$ attains
a maximum $a>0$. In either case, because $\wt{\gamma}$ is an open map,
there exists $z_0\in V_nG$ such that $|\wt{\gamma}(z_0)|>a$.
Set $b:=\wt{\gamma}(z_0)$.
The set $K:=\{(v,g)\in \wb{V_n}\times G\colon \wt{\gamma}(vg)= b\}$
is compact. After replacing $z_0$ with $v_0g_0$ for suitable $(v_0,g_0)\in K$,
we may assume that $z_0$ is of the form $v_0g_0$
with $\rho:=d_\C(v_0)=\min\{d_\C(v)\colon (v,g)\in K\}>0$.
Then $W:=\{z\in V_nG\colon \wt{\gamma}(z)\not=b\}$
is an open subset of $V_nG$ such that $G\sub W$,
and
\[
\theta\colon W\to \C,\quad \theta(z):=\frac{\wt{\gamma}(z)}{\wt{\gamma}(z)-b}
\]
is a $\C$-analytic function. Set $B_\rho:=\{z\in G_\C\colon d_\C(z)<\rho\}$.
Then $B_\rho G\sub W$,
by minimality of $d_\C(v_0)$. Also $\rho<\frac{1}{n}$
(as $z_0\in V_nG$). Let $k\in \N$ such that $\frac{1}{k}<\rho$
(and thus $k>n$). Then $V_kG\sub W$.
We show that $\eta:=\theta|_G\in \cA_k(G)$ but $\eta\not\in \cA_n(G)$.
Let $K\sub V_k$ be compact. Since $|\wt{\gamma}(z^{-1}g)|\leq \|\wt{\gamma}\|_{K,1}e^{-d(g)}$
and $d(g)\to\infty$ as $g\to\infty$, there exists a compact subset $L\sub G$
such that
\[
(\forall z\in K, g\in G\setminus L)\quad |\wt{\gamma}(z^{-1}g)|\;\leq \; a.
\]
Hence $|\wt{\gamma}(z^{-1}g)-b|\geq |b|-|\wt{\gamma}(z^{-1}g)|\geq |b|-a>0$
and thus, for each $N\in \N_0$,
\[
|\theta(z^{-1}g)|e^{Nd(g)}
\,\leq\, \frac{|\wt{\gamma}(z^{-1}g)|e^{Nd(g)}}{|b|-a}\,\leq\, \frac{\|\wt{\gamma}\|_{K,N}}{|b|-a}
\]
for all $z\in K$ and $g\in G\setminus L$. Since
$|\theta(z^{-1}g)|^{Nd(g)}$ is bounded for $(z,g)$ in the compact set $K\times L$,
we deduce that $\|\theta\|_{K,N}<\infty$.
Hence $\wt{\eta}:=\theta|_{V_kG}\in \cA_k(G)$ and
$\eta:=\wt{\eta}|_G\in \cA_k(G)$.\\[2.4mm]
$\eta\not\in \cA_n(G)$: If $\eta$ was in $\cA_n(G)$, we could find
$\sigma\in \wt{\cA}_n(G)$ with $\sigma|_G=\eta$.
Then $\theta|_{B_\rho G}=\sigma|_{B_\rho G}$,
as $B_\rho G$ is a connected open set in~$G_\C$,
and $\theta$ coincides with $\sigma$ on the
totally real submanifold $G$ of $B_\rho G$.
Given $\ve\in \,]0,\rho[$, let $c_\ve\colon [0,1]\to G_\C$ be a piecewise $C^1$-path
with $c_\ve(0)=1$ and $c_\ve(1)=v_0$,
of length $< d_\C(v_0)+\ve=\rho+\ve$.
Let $t_\ve\in [0,1]$ such that $c_\ve|_{[t_\ve,1]}$ has length~$\ve$.
Then
\begin{equation}\label{henceconv}
{\bf d}_\C(c_\ve(t_\ve),v_0)
\, =\, {\bf d}_\C(c_\ve(t_\ve),c_\ve(1))\,\leq\, \ve\,.
\end{equation}
Likewise, $d_\C(c_\ve(t_\ve))={\bf d}_\C(c_\ve(t_\ve),1)$ is bounded by the
length of $c_\ve|_{[0,t_\ve]}$ and hence $<\rho+\ve-\ve=\rho$.
Hence $c_\ve(t_\ve)\in B_\rho$ and thus
\[
\theta(c_\ve(t_\ve)g_0)=\sigma(c_\ve(t_\ve)g_0)\to \sigma(v_0g_0)=\sigma(z_0)
\]
as $\ve\to 0$ (noting that $c_\ve(t_\ve)\to v_0$ by (\ref{henceconv})).
But $|\theta(z)|\to\infty$ as $z\in W$ tends to~$z_0$, contradiction.
\end{proof}
{\bf Proof of Proposition~E.}
Each step $H_n:=\cA_n(G)\times E$ is
metrizable. For each $n\in \N$, there is $m>n$ such that
$H_n\not=H_m$ as a set (by Lemma~\ref{notstat}).
Hence condition~(a) in Corollary~\ref{yamcor} is satisfied.
Also (b) is satisfied:
Given $n$ and a $0$-neighbourhood $U\sub H_n$,
we cannot find $m\geq n$ such that
the closure $\wb{U}$ of $U$ in $H_m$ is
compact, because $(\{0\}\times E)\cap \wb{U}$
would be a compact $0$-neighbourhood in $\{0\}\times E\isom E$
then and thus~$E$ finite-dimensional (contradiction).
To verify (c), let $K\sub \cA(G)\times E$
be a sequentially compact set (with respect to the
locally convex direct limit topology).
Then the projections $K_1$ and $K_2$ of~$K$ to the factors
$\cA(G)$ and $E$, respectively, are sequentially compact sets.
Since $E$ is metrizable, $K_2\sub E$ is compact.
Now, the sequentially compact set $K_1\sub \cA(G)$ is bounded (Lemma~\ref{boundd}).
Because the locally convex direct limit $\cA(G)=\dl\,\cA_n(G)$
is regular \cite[Theorem~B.1]{GKS},
it follows that $K_1\sub \cA_n(G)$
for some $n\in \N$, and $K_1$ is bounded
in $\cA_n(G)$.
As $\cA_n(G)$ is a Montel space (Lemma~\ref{Montel}),
it follows that $K_1$ has compact closure $\wb{K_1}$ in $\cA_n(G)$.
Now $K$ is a sequentially compact subset of the compact
metrizable set $\wb{K_1}\times K_2\sub \cA(G)\times E$,
hence compact in the induced topology.
As $\cA_n(G)\times E$ and $\cA(G)\times E$
induce the same topology on the compact set $\wb{K_1}\times K_2$,
it follows that~$K$
is also compact in $\cA_n(G)\times E$.
Thus all conditions of Corollary~\ref{yamcor}
are satisfied, and thus $\cO_{DL}\not=\cO_{lcx}$.\,\Punkt
\section{{\boldmath$(\cA(G),\cdot)$} as a topological algebra}\label{pointmult}
If $n,m\in \N$, $\gamma\in \cA_n(G)$ and $\eta\in \cA_m(G)$,
then the pointwise product $\wt{\gamma}\cdot\wt{\eta}$
of the complex analytic extensions is defined on $V_kG$
with $k:=n\vee m:=\max\{n,m\}$. If $K\sub V_k$ is a compact subset
and $N,M\in \N_0$,
then $|\wt{\gamma}\wt{\eta}(z^{-1}g)|e^{(N+M)d(g)}=|\wt{\gamma}(z^{-1}g)|e^{Nd(g)}|\wt{\eta}(z^{-1}g)|e^{Md(g)}$
for all $z\in K$, $g\in G$ and thus
\begin{equation}\label{split}
\|\wt{\gamma}\cdot\wt{\eta}\|_{K,N+M}\leq \|\wt{\gamma}\|_{K,N}\|\wt{\eta}\|_{K,M}\;<\;\infty,
\end{equation}
whence $\wt{\gamma}\cdot\wt{\eta}\in \wt{\cA}_k(G)$
and hence $\gamma\cdot\eta\in \cA_k(G)$.
Thus pointwise multiplication makes $\cA(G)$ an algebra.\\[2.4mm]
To see that the multiplication is continuous at $(0,0)$,
let $W\sub \cA(G)$ be a $0$-neighbourhood.
There are $0$-neighbourhoods $W_n\sub \cA_n(G)$
such that $\sum_{n\in \N} W_n\sub W$.
We have to find $0$-neighbourhoods $Q_n\sub \cA_n(G)$
such that
\[
\sum_{(n,m)\in \N^2}Q_n\cdot Q_m\; \sub\;  W\,.
\]
This will be the case if we can achieve that
\begin{equation}\label{strategy}
(\forall k \in \N)\;\; \sum_{n\vee m=k}Q_n\cdot Q_m\;\sub\; W_k.
\end{equation}
We may assume that $W_n=\{\gamma\in \cA_n(G)\colon \|\gamma\|_{K_n,N_n}<\ve_n\}$
for some compact subset $K_n\sub V_n$, $N_n\in \N_0$
and $\ve_n\in\,]0,1]$.
After replacing $K_n$ with $K_n\cup\wb{V_{n+1}}$,
we may assume that $K_n\supseteq V_{n+1}$
and thus $K_n\supseteq K_{n+1}$, for each $n\in \N$. Thus
\[
K_1\supseteq K_2\supseteq K_3\supseteq\cdots\,.
\]
Then the $0$-neighbourhoods
\[
Q_n:=\Big\{\gamma\in \cA_n(G)\colon \|\gamma\|_{K_n,N_n}<\frac{\ve_n}{n^2}\Big\}
\]
satisfy (\ref{strategy}). To see this, let $k\in \N$
and $(n,m)\in \N^2$ such that
$n\vee m =k$. If $n=k$, using (\ref{split}) and $K_n\sub K_m$ we estimate
\begin{eqnarray*}
\|\gamma\cdot\eta\|_{K_k,N_k} &=&
\|\gamma\cdot\eta\|_{K_n,N_n}\, =\, \|\gamma\|_{K_n,N_n}\|\eta\|_{K_n,0}\, \leq\,
\|\gamma\|_{K_n,N_n}\|\eta\|_{K_m,0}\\
&<& \frac{\ve_n\ve_m}{n^2m^2}
\, \leq \, \frac{\ve_n}{n^2}=\frac{\ve_k}{k^2}.
\end{eqnarray*}
Likewise,
$\|\gamma\cdot\eta\|_{K_k,N_k}\leq \|\gamma\|_{K_n,0}\|\eta\|_{K_m,N_m}<\frac{\ve_k}{k^2}$
if $m=k$.
Since there are at most $k^2$ pairs $(n,m)$ with $n\vee m=k$,
for all choices of $\gamma_{n,m}\in Q_n$, $\eta_{n,m}\in Q_m$ the triangle inequality yields
\[
\Big\|\sum_{n\vee m =k}\gamma_{n,m}\cdot\eta_{n,m}\Big\|_{K_k,N_k}
\; < \; k^2\, \frac{\ve_k}{k^2}
\]
and thus $\sum_{n\vee m =k}\gamma_{n,m}\cdot\eta_{n,m}\in W_k$, verifying~(\ref{strategy}).
\appendix
\section{Proofs of the lemmas in Sections \ref{secprel} and \ref{preC}}\label{appA}
{\bf Proof of Lemma~\ref{specsit}.}
For fixed $y\in U_2$, the map
$s\colon U_1\times U_1\to F$,
$s(x_1,x_2):=g(x_1,y)\pi(y,h(x_2))$ is $C^{1,0}$ and $C^{0,1}$ and hence~$C^1$.
By linearity, $ds((x_1,x_2),.)$ is the sum of the partial differentials
and hence given by
\[
ds((x_1,x_2),\!(u_1,u_2))\!=\!d^{(1,0)}\!g(x_1,y,u_1)\pi(y,h(x_2))+
g(x_1,y)d^{(0,1)}\!\pi(y,dh(x_2,u_2))
\]
for all $x_1,x_2\in U_2$ and $u_1,u_2\in E_1$.
Thus $d^{(1,0)}f(x,y,u)$ exists for all\linebreak
$(x,y,u)\in U_1\times U_2\times E_1$ and
is given by
\[
d^{(1,0)}f(x,y,u)=g_1((x,u),y)\pi(y,h(x))+g(x,y)\pi(y,dh(x,u))
\]
with $g_1((x,u),y):=d^{(1,0)}g(x,y,u)$. Set $f_1((x,u),y):=d^{(1,0)}f(x,y,u)$.
By induction, $f_1\colon (U_1\times E_1)\times U_2 \to F$
is $C^{k,0}$, whence
$d^{(j+1,0)}f(x,y,u,u_1,\ldots, u_j)=$\linebreak
$d^{(j,0)}f_1((x,u),y,(u_1,0),\ldots,(u_j,0))$
exists for all $j\in \N_0$ with $j\leq k$ and $u_1,\ldots, u_j\in E_1$,
and is continuous in $(x,y,u,u_1,\ldots, u_j)$. Thus $f$ is $C^{k+1,0}$.\,\vspace{2.7mm}\Punkt

\noindent
Direct sums of locally convex spaces are always endowed with the locally convex direct sum topology
in this article (as in \cite{Bou}; see also~\cite{Jar}).
To enable the proof of Lemma~\ref{lemB1}, we shall need the following fact.
\begin{la}\label{embdisum}
Let $E$ be a locally convex space, $r\in \N_0\cup\{\infty\}$,
$M$ be a paracompact, finite-dimensional $C^r$-manifold
and $(U_j)_{j\in J}$ be a locally finite cover of~$M$
by relatively compact, open sets~$U_j$.
Then the following map is linear and a topological embedding{\rm:}
\begin{equation}\label{Psi}
\Psi\colon C^r_c(M,E)\to\bigoplus_{j\in J}C^r(U_j,E),\quad \Psi(\gamma)=(\gamma|_{U_j})_{j\in J}\,.
\end{equation}
\end{la}
\begin{proof}
The linearity is clear.
If $K\sub M$ is a compact subset, then $J_0:=$\linebreak
$\{j\in J\colon K\cap U_j\not=\emptyset\}$
is finite. The restriction~$\Psi_K$ of~$\Psi$ to $C^r_K(M,E)$
has image in $\bigoplus_{j\in J_0}C^r(U_j,E)\cong\prod_{j\in J_0}C^r(U_j,E)$
and is continuous as its components $C^r_c(M,E)\to C^r(U_j,E)$, $\gamma\mto\gamma|_{U_j}$
are continuous (cf.\ \cite[Lemma~3.7]{QUO}). Since $C^r_c(M,E)=\dl\,C^r_K(M,E)$\vspace{-.3mm} as
a locally convex space, it follows that~$\Psi$ is continuous.
Now pick a $C^r$-partition of unity $(h_j)_{j\in J}$ with $K_j:=\Supp(h_j)\sub U_j$.
Then each $m_{h_j}\colon C^r(M,E)\to C^r_{K_j}(M,E)$, $\gamma\mto h_j\cdot\gamma$
is continuous linear (e.g., as a special case of
\cite[Proposition~4.16]{ZOO}) and hence also the map $\mu \colon \bigoplus_{j\in J} C^r(U_j,E)\to
\bigoplus_{j\in J}C^r_{K_j}(U_j,E)$, $(\gamma_j)_{j\in J}\mto (h_j\gamma_j)_{j\in J}$.
Since $\mu\circ \Psi$ is an embedding \cite[Lemma~1.3]{PJM},
also $\Psi$ is a topological embedding.
\end{proof}
We also use a tool from \cite{SEM},
which is a version of \cite[Proposition~7.1]{MEA} with
parameters in a set~$U$ (for
countable~$J$, see already \cite[Proposition~6.10]{ZOO}):
\begin{la}\label{sumpardep}
Let $X$ be a finite-dimensional vector space, $U \sub X$ open,
$(E_j)_{j\in J}$ and $(F_j)_{j\in J}$ families of locally convex spaces,
$U_j\sub E_j$ open, $r\in \N_0\cup\{\infty\}$
and $f_j\colon U\times U_j\to F_j$ be a map. Assume that there is a finite set $J_0\sub J$ such that
$0\in U_j$ and $f_j(x,0)=0$ for all $j\in J\setminus J_0$ and $x\in U$.
Then $\bigoplus_{j\in J}U_j:=(\bigoplus_{j\in J}E_j)\cap \prod_{j\in J}U_j$
is open in $\bigoplus_{j\in J}E_j$, and we can consider
\[
f\colon U\times \bigoplus_{j\in J}U_j\to\bigoplus_{j\in J}F_j,\quad
f(x,(x_j)_{j\in J}):=(f_j(x,x_j))_{j\in J}.
\]
\begin{itemize}
\item[\rm(a)]
If $J$ is countable and each $f_j$ is $C^r$, then
$f$ is $C^r$.
\item[\rm(b)]
If $J$ is uncountable and each $f_j$ is $C^{r+1}$,
then $f$ is~$C^r$.
\end{itemize}
The conclusion of {\rm(b)} remains valid if each~$f_j$ is $C^{0,1}$
and the mappings $f_j$ and $d^{(0,1)}f_j \colon U\times U_j\times E_j\to F_j$
are $C^r$.\,\Punkt
\end{la}
{\bf Proof of Lemma~\ref{lemB1}.}
Given $g_0\in G$, let $U\sub G$ be a relatively compact, open neighbourhood of~$g_0$.
We show that $\pi_U\colon U\times C^\infty_c(G)\to C^\infty_c(G)$, $(g,\gamma)\mto\pi(g,\gamma)$ is smooth.
To this end, let $(U_j)_{j\in J}$ be a locally finite cover of~$G$ by relatively compact,
open sets~$U_j$. Then also $(U^{-1}U_j)_{j\in J}$ is locally finite.\footnote{If $K\sub G$ is compact,
then $U^{-1}U_j\cap K\not=\emptyset$ $\Leftrightarrow$
$U_j\cap UK\not=\emptyset$.}
As a consequence, both
$\Psi\colon C^\infty_c(G)\to\bigoplus_{j\in J}C^\infty(U_j)$, $\Psi(\gamma):=(\gamma|_{U_j})_{j\in J}$
and the corresponding restriction map $\Theta\colon C^\infty_c(G)\to \bigoplus_{j\in J}C^\infty(U^{-1}U_j)$
are linear topological embeddings (Lemma~\ref{embdisum}).
Since
\[
\im(\Psi)=\{(\gamma_j)_{j\in J}\colon (\forall i,j\in J)(\forall x\in U_i\cap U_j)\;\gamma_i(x)=\gamma_j(x)\}
\]
is a closed vector subspace of $\bigoplus_{j\in J}C^\infty(U_j)$,
the map $\pi_U$ will be smooth if we can show that $\Psi\circ \pi_U$ is smooth
(cf.\ \cite[Lemma~10.1]{BGN}).
For each $j\in J$, the evaluation map $\ve_j\colon C^\infty(U^{-1}U_j)\times U^{-1}U_j\to\C$,
$\ve_j(\gamma,x):=\gamma(x)$ is smooth
(see \cite{GaN} or \cite[Proposition~11.1]{ZOO}).
Lemma~\ref{explaw} shows that
\[
\Xi_j\colon U\times C^\infty(U^{-1}U_j)\to C^\infty(U_j), \quad \Xi_j(g,\gamma)(x):=\gamma(g^{-1}x)
\]
is $C^\infty$, as $\wh{\Xi_j}\colon U\times C^\infty(U^{-1}U)\times U_j\to \C$,
$\wh{\Xi_j}(g,\gamma,x):=\gamma(g^{-1}x)=\ve_j(\gamma,g^{-1}x)$ is smooth.
Then
\[
\Xi \colon U\times \bigoplus_{j\in J}C^\infty(UU_j)\to \bigoplus_{j\in J}C^\infty(U_j),\quad
\Xi(x,(\gamma_j)_{j\in J}):=(\Xi_j(x,\gamma_j))_{j\in J}
\]
is $C^\infty$, by Lemma~\ref{sumpardep}.
Hence $\Psi\circ \pi_U=   \Xi \circ (\id_U\times \Theta)$ (and hence $\pi_U$) is $C^\infty$.\,\vspace{2.4mm}\Punkt

\noindent
{\bf Proof of Lemma~\ref{lemB2}.}
Since $C^0_c(M)=\dl\,C^0_K(M)$
as a locally convex space, the linear map $m_f$ will be continuous
if $C^0_K(M)\to C^0_K(M,E)$,
$\gamma\mto \gamma f$ is continuous. This is the case by \cite[Lemma~3.9]{QUO}.\,\vspace{2.4mm}\Punkt

\noindent
{\bf Proof of Lemma~\ref{pardep}.}
It suffices to prove the lemma for $r\in \N_0$.
By \cite[Lemma~A.2]{PJM}, $g$ is continuous.
If $r>0$, $k\in \N_0$ with $k\leq r$, $p\in P$ and $q_1,\ldots, q_k\in X$,
there is $\ve>0$ such that
\[
h(t_1,\ldots, t_k):=g(p+{\textstyle\sum_{j=1}^k} t_kq_j)
\]
is defined for $(t_1,\ldots, t_k)$ in some open $0$-neighbourhood $W\sub \R^k$.
By \cite[Lemma~A.3]{PJM}, $h\colon W\to E$ is $C^k$,
and $d^{(k,0)}g(x,p,q_1,\ldots, q_k)=\partial^{(1,\ldots,1)}h(0,\ldots, 0)$\linebreak
$=
\int_K(D_{(q_k,0)}\cdots D_{(q_1,0)}f)(p,x)\,d\mu(x)=
\int_Kd^{(k,0)}f(p,x,q_1,\ldots, q_k)\, d\mu(x)$. By the case $r=0$,
the right hand side is continuous in $(p,q_1,\ldots, q_k)$.
So~$g$~is~$C^r$.\,\vspace{2.4mm}\Punkt

\noindent
{\bf Proof of Lemma~\ref{csuppE}.}
Let $K:=\Supp(\gamma)\sub G$.
For $g\in G$, we have $\pi_w(g)=\pi(g,w)
=\int_G\gamma(y)\pi(g,\pi(y,v)\,d\lambda_G(y)
=\int_G\gamma(y)\pi(gy,v)\,d\lambda_G(y)
=\int_G\gamma(g^{-1}y)\pi(y,v)\,d\lambda_G(y)$,
using left invariance of Haar measure for the last equality.
Given $g_0\in G$, let $U\sub G$ be an open, relatively compact
neighbourhood of~$g_0$.
As $g^{-1}y\in K$ implies $y\in \wb{U}K$
for $g\in U$ and $y\in G$,~get
\[
\pi_w(g)=\int_{\wb{U}K}\gamma(g^{-1}y)\pi(y,v)\,d\lambda_G(y)\quad\mbox{for all $g\in U$.}
\]
Since $U\times \wb{U}K\to E$, $(g,y)\mto
\gamma(g^{-1}y)\pi(y,v)$ is a $C^{\infty,0}$-map,
Lemma~\ref{pardep} shows that $\pi_w|_U$ is smooth.
Hence $\pi_w$ is smooth and thus $w\in E^\infty$ indeed.
Testing equality with continuous linear functionals
and using Fubini's theorem and then left invariance of Haar measure,
one verifies that
\begin{eqnarray*}
\Pi(\gamma*\eta,v)
&=&\int_G\int_G\gamma(z)\eta(z^{-1}y)\pi(y,v)\,d\lambda_G(z)d\lambda_G(y)\\
&=&\int_G\gamma(z) \int_G\eta(z^{-1}y)\pi(y,v)\,d\lambda_G(y)d\lambda_G(z)\\
&=& \int_G\gamma(z) \int_G\eta(y)\pi(zy,v)\,d\lambda_G(y)d\lambda_G(z)\\
&=&\int_G\gamma(z) \pi(z,\Pi(\eta,v))\, d\lambda_G(z)\,=\,\Pi(\gamma,\Pi(\eta,v)).
\end{eqnarray*}
Hence $E$ (and $E^\infty$) are $C^\infty_c(G)$-modules.\,\vspace{2.4mm}\Punkt

\noindent
{\bf Proof of Lemma~\ref{extGC}.}
If $\pi_v$ has a $\C$-analytic extension $\wt{\pi}_v$ to $GV$ for some open identity neighbourhood
$V\sub G_\C$, then (like any $\C$-analytic map) $\wt{\pi}_v$ is $\R$-analytic~\cite{GaN}.
As inclusion $j\colon G\to G_\C$ is $\R$-analytic,
so is $\pi_v=\wt{\pi}_v\circ j$.\\[2.4mm]
Conversely, assume that $\pi_v$ is $\R$-analytic. There is
an open $0$-neighbourhood $W\sub L(G)_\C$
such that $\phi:=\exp_{G_\C}|_W$ is a $\C$-analytic diffeomorphism
onto an open subset $\phi(W)$ in~$G_\C$,
$\phi(W\cap L(G))=G\cap \phi(W)$,
and $\psi:=\phi|_{W\cap L(G)}$
is an $\R$-analytic diffeomorphism onto its image in~$G$.
Then $\pi_v\circ \psi$
is $\R$-analytic and hence extends to a $\C$-analytic map
$f\colon \wt{W}\to E$ for some open set $\wt{W}\sub W$
containing $W\cap L(G)$,
and thus $\phi(\wt{W})\to E$, $z\mto f(\phi^{-1}(z))$
is a $\C$-analytic extension of $\pi_v|_{W\cap L(G)}$.
We now find $n\in \N$ such that $V_n\sub \phi(\wt{W})$
and $U_n\sub \im(\psi)$,
using the notation from \cite{GKS}.
Hence $v\in \wt{E}_n$ and hence $v\in E_{4n}$,
by \cite[Lemma~3.2]{GKS}.\,\vspace{2.4mm}\Punkt

\noindent
{\bf Proof of Lemma~\ref{betterK}.}
For each $k\in K$, there are $g_k\in G$ and $v_k\in V_n$
such that $k=g_kv_k$. Let $P_k\sub V_n$ be a compact
neighbourhood of $v_k$. Then $(g_kP_k^0)_{k\in K}$
is an open cover of $K$, whence there exists a finite subset $F\sub K$
such that $K\sub\bigcup_{k\in F}g_kP_k$. Then $P:=\bigcup_{k\in F}P_k$
is a compact subset of $V_n$ and $GK\sub GP$.\,\vspace{2.4mm}\Punkt

\noindent
{\bf Proof of Lemma~\ref{changeK}.}
If $z\in K$, then $z=h\ell$ for some $h\in G$ and $\ell \in L$.
Then $h=z\ell^{-1}\in KL^{-1}$.
For $g\in G$, we have
\[
|\gamma(z^{-1}g)|e^{Nd(g)}=|\gamma(\ell^{-1}(h^{-1}g))|e^{Nd(g)}
\leq e^{Nd(h)}|\gamma(\ell^{-1}(h^{-1}g))|e^{Nd(h^{-1}g)}
\]
as $d(g)=d(h(h^{-1}g))\leq d(h)+d(h^{-1}g)$. The assertion follows.\,\vspace{2.4mm}\Punkt

\noindent
{\bf Proof of Lemma~\ref{cxpardep}.}
Let $K_1\sub K_2\sub\cdots$ be compact subsets of~$Y$ such that $Y=\bigcup_{n\in \N}K_n$.
Then $g_n(z):=\int_{K_n} f(z,y)\,d\mu(y)$ exists for all $z\in U$ \cite[1.2.3]{Her}.
By Lemma~\ref{pardep}, the map
$g_n\colon U\to E$ is $C^1$ with
$dg_n(z,w)=\int_{K_n}d^{(1,0)}f(z,y,w)\,d\mu(y)$,
which is $\C$-linear in $w\in Z$.
As~$E$ is sequentially complete, this implies
that $g_n$ is $\C$-analytic \cite[1.4]{STU}.
For each continuous seminorm $q$ on~$E$,
we have $\int_Y q(f(z,y))\,d\mu(y)\leq\int_Y m_q(y)\,d\mu(y)<\infty$.
Since $\lim_{n\to\infty}\int_{K_n}m_q(y)\,d\mu(y)=\int_Y m_q(y)\,d\mu(y)$,
given $\ve>0$ there exists $N\in \N$
such that $\int_{Y\setminus K_n}m_q(y)\,d\mu(y)<\ve$ for all $n\geq N$.
Hence
\begin{equation}\label{cauchy1}
q(g_\ell(z)-g_n(z))\leq \int_{K_\ell\setminus K_n}m_q(y)\,d\mu(y)\;<\;\ve
\end{equation}
for all $\ell \geq n\geq N$, showing that $(g_n(z))_{n\in \N}$ is a
Cauchy sequence in~$E$ and hence convergent to some
element $g(z)\in E$. For each continuous linear functional $\lambda\colon E\to\C$,
we have $|\lambda(f(z,y))|\leq m_{|\lambda|}(y)$, whence the function $|\lambda(f(z,.))|$ is $\mu$-integrable
and $\int_Y\!\lambda(f(z,y))d\mu(y)
\!=\!\lim_{n\to\infty}\!\int_{K_n}\!\! \lambda(f(z,y))d\mu(y)$\linebreak
$=\lim_{n\to\infty}\lambda(g_n(z)) 
=\lambda(\lim_{n\to\infty}g_n(z))=\lambda(g(z))$.
Hence $g(z)$ is the weak integral $\int_Yf(z,y)\,d\mu(y)$.
As $\int_Y q(f(z,y))\,d\mu(y)\leq \int_Y m_q(y)\,d\mu(y)<\infty$,
the integral $\int_Yf(z,y)\,d\mu(y)$ is absolutely convergent.
Letting $\ell \to\infty$ in (\ref{cauchy1}), we see that
$q(g(z)-g_n(z))\leq \ve$ for all $z\in U$ and $n\geq N$.
Thus $g_n\to g$ uniformly.
Since $E$ is sequentially complete,
the uniform limit~$g$ of $\C$-analytic functions is $\C$-analytic \cite[Proposition~6.5]{BaS},
which completes the proof.\,\vspace{2.4mm}\Punkt

\noindent
{\bf Proof of Lemma~\ref{intrep1}}
(see \cite[4.3 on p.\,1592]{GKS} for
an alternative argument).
Given $z_0\in V_n$ and $x\in G$,
let $K\sub V_n$ be a compact neighbourhood
of $z_0$.
If $q$ is a $G$-continuous seminorm on~$E$, then there exist
$C_q\geq 0$ and $m\in \N_0$ such that $q(\pi(y,v))\leq q(v)C_qe^{md(y)}$
(see (\ref{good2})).
Choose $\ell\in \N_0$ such that $C:=\int_Ge^{-\ell d(y)}\,d\lambda_G(y)<\infty$
(see (\ref{secd})).
For $N\in \N_0$ with $N\geq m+\ell$, we obtain,
using that
$d(y)=d(xx^{-1}y)\leq d(x)+d(x^{-1}y)$,
\begin{eqnarray}
q(\wt{\gamma}(z^{-1}y)\pi(y,v)) & \leq & |\wt{\gamma}(k^{-1}x^{-1}y)|q(\pi(y,v))\notag \\
&\leq & |\wt{\gamma}(k^{-1}x^{-1}y)|e^{Nd(x^{-1}y)}C_q e^{(m-N)d(y)}e^{Nd(x)}q(v)\notag\\
&\leq&  C_qe^{Nd(x)}\|\gamma\|_{K,N}e^{-\ell d(y)}q(v) \,\,\label{dochuse}
\end{eqnarray}
for all $z=xk$ with $k \in K$, and all $y\in G$.
Hence Lemma~\ref{cxpardep} shows that the integral in (\ref{integals})
converges absolutely for all $z\in x K^0$
and defines a $\C$-analytic function $xK^0\to E$.
Since $xz_0\in GV_n$ was arbitrary,
the integral in~(\ref{integals}) exists
for all $z\in GV_n$ and defines a $\C$-analytic function $\eta \colon GV_n\to E$.
For $x\in G$,
\begin{eqnarray*}
\pi(x,w)&=&
\pi(x,.)\left(\int_G\gamma(y)\pi(y,v)\,d\lambda_G(y)\right)
\,=\,\int_G\gamma(y)\pi(x,\pi(y,v))\,d\lambda_G(y)\\
&=& \int_G\gamma(y)\pi(xy,v)\,d\lambda_G(y)
\,=\,\int_G\gamma(x^{-1}y)\pi(y,v)\,d\lambda_G(y)\,=\,\eta(x)
\end{eqnarray*}
by left invariance of Haar measure.
Hence $\eta$ is a $\C$-analytic extension of $\pi_w$
to $GV_n$, whence $w\in E_n$ and $\wt{\pi_w}=\eta$.\,\vspace{2.4mm}\Punkt

\noindent
{\bf Proof of Lemma~\ref{sin}.}
Since $L(G)$ is a compact Lie algebra,
there exists a positive definite bilinear form $\langle.,.\rangle \colon L(G)\times L(G)\to \R$
making $e^{\ad(x)}=\Ad(\exp_G(x))$ an isometry for each $x\in L(G)$.
Since $G$ is generated by the exponential image, it follows that
$\Ad(g)$ is an isometry for each $g\in G$.
Now use the same symbol, $\langle.,.\rangle$,
for the unique extension to a hermitean form
$L(G)_\C\times L(G)_\C\to\C$. Write $B_r\sub L(G)_\C$
for the open ball of radius $r$ around $0$.
After replacing the form by a positive multiple
if necessary, we may assume that $\exp_{G_\C}$
restricts to a homeomorphism $\phi$ from $B_1$
onto a relatively compact, open subset of $G_\C$.
Then the sets $V_n:=\exp_{G_\C}(B_{1/n})$
form a basis of relatively compact, connected open
identity neighbourhoods, such that $\wb{V_{n+1}}\sub V_n$
and $gV_ng^{-1}=\exp_{G_\C}(\Ad(g)(B_{1/n}))
=\exp_{G_\C}(B_{1/n})=V_n$.
If $K\sub V_n$ is compact, then $A:=\phi^{-1}(K)$
is a compact subset of $B_{1/n}$ and thus $r:=\max\{\sqrt{\langle x,x\rangle}\colon x\in A\}<1/n$.
Then $\exp_{G_\C}(\wb{B_r})$
is a compact, conjugation-invariant subset of $G$
which contains $K$ and thus
$\wb{\{gxg^{-1}\colon g\in G,\,x\in K\}}\sub
\exp_{G_\C}(\wb{B_r})\sub V_n$.\vspace{2.4mm}\,\Punkt

\noindent
{\bf Proof of Lemma~\ref{intrnew}.}
Let $x_0\in G$, $z_0\in V_n$ and $K\sub V_n$
be a compact neighbourhood of~$y_0$.
Then $K_1:=\wb{\{gzg^{-1}\colon g\in G,\,z\in K\}}\sub V_n$
is compact, by choice of~$V_n$.
If $q$ is a $G$-continuous seminorm on~$E$, then there exist
$C_q\geq 0$ and $m\in \N_0$ such that $q(\pi(y,v))\leq q(v)C_qe^{md(y)}$
(see (\ref{good2})).
Then $\|v\|_{K_1,q}:=\sup q(\wt{\pi}_v(K_1))<\infty$.
Choose $\ell\in \N_0$ such that $C:=\int_Ge^{-\ell d(y)}\,d\lambda_G(y)<\infty$
(see (\ref{secd})).
Note that $\wt{\pi}_v(xzy)=\wt{\pi}_v(xyy^{-1}zy)=\pi(xy,\wt{\pi}_v(y^{-1}zy))$
for all $x\in G$, $z \in K$ and $y\in G$, where
$y^{-1}zy\in K_1$. Thus
\begin{eqnarray}
q(\gamma(y)\wt{\pi}_v(xzy)) &=& |\gamma(y)|q(\pi(xy,\wt{\pi}_v(y^{-1}zy))\,=\,
|\gamma(y)|C_qe^{md(xy)}q(\wt{\pi}_v(y^{-1}zy))\notag\\
&\leq&  C_q\|v\|_{K_1,q}|\gamma(y)|e^{md(y)}e^{md(x)}\notag\\
&\leq&
C_q\|v\|_{K_1,q}e^{md(x)}\|\gamma\|_{m+\ell}e^{-\ell d(y)},\,\,\label{dochunew}
\end{eqnarray}
using the notation from (\ref{reusno}).
Hence Lemma~\ref{cxpardep} shows that the integral in (\ref{integnew})
converges absolutely for all $z\in x K^0$
and defines a $\C$-analytic function $xK^0\to E$.
Notably, this holds for $x=x_0$.
Since $x_0z_0\in GV_n$ was arbitrary,
the integral in~(\ref{integnew}) exists
for all $z\in GV_n$ and defines a $\C$-analytic map $\eta \colon GV_n\to E$.
For $x\in G$, we have
$\pi(x,w)=
\int_G\gamma(y)\pi(x,\pi(y,v))\,d\lambda_G(y)=\eta(x)$.
Hence $\eta$ is a $\C$-analytic extension of $\pi_w$
to $GV_n$, and thus $w\in E_n$ and $\wt{\pi}_w=\eta$.\,\vspace{2.4mm}\Punkt

\noindent
{\bf Proof of Lemma~\ref{reach}.}
We need only show that $\Pi$ is separately continuous.
In fact, $\cA(G)$ is barrelled,
being a locally convex direct limit of the Fr\'{e}chet spaces $\cA_n(G)$
\cite[II.7.1 and II.7.2]{SaW}.
Hence, if $\Pi$ is separately continuous, it
automatically is hypocontinuous in its second argument
\cite[II.5.2]{SaW} and hence sequentially continuous
(see \cite[p.\,157, Remark following \S40, 1., (5)]{Koe}).\\[2.4mm]
Let $\Pi_n\colon \cA_n(G)\times E\to E^\omega$ be the restriction of
$\Pi$ to $\cA_n(G)\times E$. Then $\Pi_n$ is continuous (see (\ref{mpxx})).
For $\gamma\in \cA(G)$, there exists $n\in \N$ such that $\gamma\in \cA_n(G)$.
Thus $\Pi(\gamma,.)=\Pi_n(\gamma,.)\colon E\to E^\omega$
is continuous. If $v\in E$, then the linear map $\Pi(.,v)=\dl\,\Pi_n(.,v)\colon \cA(G)\to E^\omega$
is continuous.\,\vspace{2.4mm}\Punkt

\noindent
{\bf Proof of Lemma~\ref{boundd}.}
If $K$ was unbounded,
we could find $x_1,x_2,\ldots$ in $K$
and a continuous seminorm~$q$ on~$E$
such that $q(x_n)\to \infty$ as $n\to\infty$.
Then $(x_n)_{n\in \N}$ does not have a convergent subsequence, contradiction.\,\vspace{2.4mm}\Punkt

\noindent
{\bf Proof of Lemma~\ref{Montel}.}
Since $\wt{\cA}_n(G)$ is a Fr\'{e}chet space and hence barrelled,
it only remains to show that each bounded
subset $M\sub \wt{\cA}_n(G)$ is relatively compact.
Because $\wt{\cA}_n(G)$ is complete, we need only
show that~$M$ is precompact.
Thus, for each compact set $K\sub V_n$, $N\in \N_0$
and $\ve>0$, we have to find a finite subset
$\Gamma\sub M$ such that
\begin{equation}\label{givesprec}
M\; \sub \; \bigcup_{\gamma\in \Gamma}\{\eta\in \wt{\cA}_n(G)\colon
\|\eta-\gamma\|_{K,N}\leq \ve\}.
\end{equation}
Since~$M$ is bounded, $C:=\sup\{ \|\gamma\|_{K,N+1} \colon \gamma\in M \}<\infty$.
Choose $\rho>0$ with $2Ce^{-\rho}<\ve$.
Then $K_1:=\{g\in G\colon d(g)\leq \rho\}$ is a compact subset of~$G$ (see \cite[p.\,74]{Gar}),
and hence $L:=K^{-1}K_1$ is compact in $G_\C$.
The inclusion map $\wt{\cA}_n(G)\to \cO(V_nG)$ being continuous,
$M$ is bounded also in the space $\cO(V_nG)$
of $\C$-analytic functions on the finite-dimensional
complex manifold $\cO(V_nG)$, equipped with
the compact-open topology, which is a prime example
of a Montel space.
Hence, we find a finite subset $\Gamma\sub M$
such that
\begin{equation}\label{selects}
(\forall \eta\in M)\;(\exists \gamma\in\Gamma)\;\; \|(\eta-\gamma)|_L\|_\infty \;<\; e^{-N\rho}\ve\,.
\end{equation}
Given $\eta\in M$, pick $\gamma\in \Gamma$ as in (\ref{selects}).
Let $z\in K$, $g\in G$.
If $d(g)\geq \rho$, then
\[
|\eta(z^{-1}g)-\gamma(z^{-1}g)|e^{Nd(g)}\leq (|\eta(z^{-1}g)|+|\gamma(z^{-1}g)|)e^{(N+1)d(g)}e^{-d(g)}
\leq 2Ce^{-\rho}<\ve\,.
\]
If $d(g)<\rho$, then $z^{-1}g\in L$ and thus
\[
|\eta(z^{-1}g)-\gamma(z^{-1}g)|e^{Nd(g)}\leq e^{-N\rho}\ve e^{Nd(g)}<\ve,
\]
by (\ref{selects}). Hence $\|\eta-\gamma\|_{K,N}\leq \ve$,
showing that (\ref{givesprec}) holds for $\Gamma$.\,\Punkt

\noindent
{\small
Helge  Gl\"{o}ckner, Universit\"at Paderborn, Institut f\"{u}r Mathematik,\\
Warburger Str.\ 100, 33098 Paderborn, Germany.\\[2mm]
Email: {\tt  glockner\at{}math.upb.de}}
\end{document}